\documentclass{amsart}
\usepackage{amsmath}
\usepackage{graphicx}
\usepackage{latexsym}
\usepackage{epstopdf}
\usepackage{pinlabel}

\newcommand{\co}{\mskip0.5mu\colon\thinspace}

\newtheorem{theorem}{Theorem}[section]
\newtheorem{lemma}[theorem]{Lemma}
\newtheorem{proposition}[theorem]{Proposition}
\newtheorem{corollary}[theorem]{Corollary}

\newtheorem{defn}[theorem]{Definition}

\theoremstyle{remark}
\newtheorem{remark}[theorem]{Remark}

\def\dfn{\emph}

\def\base{\bf{\Sigma}}
\def\ds{\displaystyle}

\newcommand{\R}{\mathbb{R}}

\newcommand{\whf}{\widehat{F}}
\newcommand{\whs}{\widehat{\Sigma}}

\def \bs {\backslash}
\def \x {\times}
\def \eu{{\text{e}}}

\newcommand{\CC}{{\mathbb C}}

\newcommand{\RR}{{\mathbb R}}

\newcommand{\jJ}{{\mathcal J}}
\newcommand{\uU}{{\mathcal U}}
\newcommand{\p}{\partial}
\newcommand{\defin}[1]{\textbf{#1}}
\newcommand{\paper}{_P}
\newcommand{\spine}{_\Sigma}
\newcommand{\crit}{^{\operatorname{crit}}}

\newcommand{\Lie}{{\mathcal L}}

\begin{document}

\title[Contact $3$-manifolds with arbitrarily large Stein fillings]
{Families of contact $3$-manifolds \\ with arbitrarily large Stein fillings}

\author[R. \.{I}. Baykur]{R. \.{I}nan\c{c} Baykur}
\address{Refik {I}nan\c{c} Baykur \newline 
\indent Max Planck Institut f\"ur Mathematik, Bonn, 53111, Germany \newline
\indent Department of Mathematics, Brandeis University, Waltham, MA 02453, USA}
\email{baykur@mpim-bonn.mpg.de, baykur@brandeis.edu}

\author[J. Van Horn-Morris]{Jeremy Van Horn-Morris \\ \\ (with an appendix by Samuel Lisi and Chris Wendl )}
\address{Jeremy Van Horn-Morris \newline 
\indent Department of Mathematics, Stanford University, \newline 
\indent Building 380, Stanford, CA 94305,  USA }
\email{jvanhorn@math.stanford.edu }

\address{Samuel Lisi \newline 
\indent Universit\'e Libre de Bruxelles, D\'epartement de Math\'ematiques \newline
\indent CP 218, Boulevard du Triomphe, B-1050 Bruxelles, Belgium}
 \email{samuel.lisi@ulb.ac.be}
 
\address{Chris Wendl \newline 
\indent Department of Mathematics, University College London, \newline
\indent Gower Street, London WC1E 6BT, United Kingdom}
\email{c.wendl@ucl.ac.uk}

\begin{abstract}
We show that there are vast families of contact $3$-manifolds each member of which admits infinitely many Stein fillings with arbitrarily big euler characteristics and arbitrarily small signatures ---which disproves a conjecture of Stipsicz and Ozbagci. To produce our examples, we set a framework which generalizes the construction of Stein structures on allowable Lefschetz fibrations over the $2$-disk to those over any orientable base surface, along with the construction of contact structures via open books on $3$-manifolds to spinal open books introduced in \cite{LVW}.

\end{abstract}

\maketitle

\setcounter{secnumdepth}{2}
\setcounter{section}{0}

% ================================================================================================================ 

\section{Introduction} 
% ================================================================================================================
% ================================================================================================================

Understanding the topology of possible Stein fillings of a fixed contact $3$-manifold has been an active line of research in the past couple of decades. By now it is known that there are contact $3$-manifolds which admit no Stein filling, as well as a unique Stein filling, or many, and even infinitely many ones, up to diffeomorphisms. However, all examples of Stein fillings of a fixed contact $3$-manifold known up to date bore the same curious aspect: their characteristic numbers constitute a finite set. Andras Stipsicz conjectured that the set of signatures and euler characteristics of all possible Stein fillings of a closed contact $3$-manifold is finite \cite[Conjecture 1.2]{S}. The same conjecture was also formulated by Burak Ozbagci and Andras Stipsicz for the euler characteristics alone \cite[Conjecture 1.2]{OS1}; \cite[Conjecture 12.3.16]{OS}, and more specifically as the euler characteristics being bounded above \cite[Conjecture 1.3.9]{OS}. There are many examples of Stein fillable contact structures for which the finiteness of both characteristic numbers is seen to hold true; those on $3$-manifolds which are non-flat circle bundles over orientable surfaces \cite{S}, or those which admit compatible planar open books \cite{Ka} are a few. Our main theorem however, disproves this conjecture, for all of its flavors:

\begin{theorem} \label{mainthm}
There are infinite families of contact $3$-manifolds, where each contact $3$-manifold admits a Stein filling whose euler characteristic is larger and signature is smaller than any two given numbers. 
\end{theorem}

\enlargethispage{0.5in}

Let us call a Lefschetz fibration on a $4$-manifold ``allowable'', if its base and regular fibers are connected, compact surfaces with \textit{non-empty} boundaries, and if each vanishing cycle is homologically non-trivial in the fiber. Following the works of Eliashberg and Gompf on handle decompositions of compact Stein manifolds, Loi and Piergallini, proved that any Stein domain admits a Lefschetz fibration structure \cite{LP} (and an alternative proof was later given by Akbulut and Ozbagci \cite{AO}). Moreover, the Stein structure on an allowable Lefschetz fibration can be chosen so that the contact structure it induces on the boundary agrees with the one that the Thurston-Winkelnkemper construction would hand when applied to the natural open book induced by the Lefschetz fibration on the boundary. We will use an extension of this result to Lefschetz fibrations over \textit{arbitrary compact surfaces} (that is orientable surfaces with any number of boundary components and of any genera) filling the same contact structure on the $3$-manifold boundary induced by a \textit{generalized} open book structure: roughly speaking, we will use a decomposition of a $3$-manifold as a certain ``plumbing'' of a surface bundle over disjoint union of circles and circle bundles over arbitrary surfaces, where the surfaces in the former and latter collections have the same topology, respectively. These generalized open books are introduced and studied in \cite{LVW} under the name \textit{spinal open books}, which we will adopt here. Note that when we have a surface bundle over a circle and a circle bundle over a $2$-disk, this is the usual open book decomposition of a $3$-manifold, and thus, exists on all $3$-manifolds. We prove the following theorem using handle decompositions and convex surface theory:

\begin{theorem} \label{thm2} 
If $f: X \to \Sigma$ is an allowable Lefschetz fibration with bounded fiber (where $\Sigma$ is any compact surface with non-empty boundary), then $X$ admits a Stein structure. Moreover, the Stein structures on any two allowable Lefschetz fibrations filling the same spinal open book can be chosen so that they induce the same contact structure on the boundary.
\end{theorem}

\noindent This result was known to Sam Lisi and Chris Wendl, who provide their proof, which is a variation of a technique of Gompf and Thurston, in the appendix to this article. Combined with Loi and Piergallini's stronger result on the existence of allowable Lefschetz fibrations (over the $2$-disk) on compact Stein manifolds, this theorem generalizes, in the obvious way, the characterization of Stein manifolds in terms of the Lefschetz fibrations they can be equipped with. (See Corollary~\ref{LPgeneralization}.)

The organization of our article is as follows: 

We discuss spinal open books and the natural contact structures we associate to them in Section~\ref{Spinal}. These parallel the descriptions in \cite{LVW} and in the appendix. For completeness, we show, using convex surface theory, that there is a unique choice of a compatible contact structure on a given spinal open book (Propositions \ref{prop:uniqueframed} and \ref{prop:welldefined}). Discussing the handle decompositions and induced Stein cobordisms for building an allowable Lefschetz fibration over an arbitrary compact surface with non-empty boundary, we prove Theorem~\ref{thm2} in the same section using a cut-and-paste operation we call \emph{folding} (or a \emph{spinal tap} in the case of a spinal open book). Our techniques have the same flavor as those used in \cite{Baldwin} and mimic the construction in \cite{Avdek}.

Section~\ref{ConstructLFs} is where we present our families of examples for Theorem~\ref{mainthm}. Our main examples will be the graph manifolds $Y(g,h,n)$ prescribed by the surgery diagram in Figure~\ref{graph}, for which we will define a distinguished contact structure $\xi_{Y(g,h,n)}$ via the framed spinal open book on it. Here, for each triple of integers $g \geq 2, h \geq 1, n \leq 2h-2$, we produce infinite families of Stein fillings of contact $3$-manifolds $(Y(g,h,n), \xi_{g,h,n})$, by constructing infinite families of Lefschetz fibrations, whose euler characteristics can be chosen to be arbitrarily big. We will also show that the Stein fillings of $(Y(2,h,n), \xi_{2,h,n})$ can be chosen so that they have arbitrarily small (negative) signatures. All these examples are derived from special families of Lefschetz fibrations on \textit{closed} $4$-manifolds (Theorem~\ref{LFs}), which are built using relations in the mapping class groups of surfaces with boundaries after \cite{BKM}. Lastly, we outline how to get similar families of Stein fillings of a fixed contact structure on more general $3$-manifolds, so as to illustrate that the contact $3$-manifolds $(Y(g,h,n), \xi_{g,h,n})$ above are nowhere close to being special in this sense.

\enlargethispage{0.5in}

\newpage
% ================================================================================================================
\section{Preliminaries} \label{Prelim}
% ================================================================================================================
% ================================================================================================================

Here we review the background material we will use and generalize in the later sections. All manifolds in this article are assumed to be compact, smooth and oriented, whereas the maps between them are always smooth.

% ================================================================================================================
\subsection{Lefschetz fibrations and mapping class groups} \
% ================================================================================================================

A \dfn{Lefschetz fibration} is a surjective map $f\colon\, X\to \Sigma$, where $X$ and $\Sigma$ are $4$- and $2$-dimensional compact manifolds, respectively, such that $f$ fails to be a submersion along a discrete set $C$, and around each \textit{critical point} in $C$ it conforms to the local model $f(z_1,z_2)=z_1 z_2$, compatible with orientations. If the regular fiber $F$ has genus $g$ and $\Sigma$ has genus $h$, we say that $(X,f)$ is a genus $g$ Lefschetz fibration over a genus $h$ surface. The critical points arise from attaching $2$--handles to regular fibers with framing $-1$ with respect to the framing induced by the fiber. We will refer to these $2$--handles as \dfn{Lefschetz handles}. We will assume that each \emph{singular fiber} contains only one critical point, which can be achieved after a small perturbation of any given Lefschetz fibration. When there are no critical points, $f: X \to \Sigma$ is nothing but a surface bundle over a surface, so $f$ always restricts to a surface bundle over $\Sigma \setminus f(C)$ on $X \setminus f^{-1}(f(C))$ and, in particular, over $\partial \Sigma$ on $\partial X$. The reader is advised to turn to \cite{GS} for a detailed treatment of Lefschetz fibrations via handlebody decompositions.

We will call a Lefschetz fibration \dfn{allowable}, if both the base and the regular fiber have non-empty boundaries, and if no fiber contains a \textit{closed} embedded surface. In the literature, allowable Lefschetz fibrations over the $2$-disk are called \dfn{PALFs}, ``positive allowable Lefschetz fibrations'', where positivity emphasizes the orientation preserving local model we prescribed for the Lefschetz singularities. 

Let $\Sigma_{g,r}^s$ denote a compact oriented surface of genus $g$ with $s$ boundary components and $r$ marked points in the interior. The \emph{mapping class group\,}, $\Gamma_{g,r}^s$, of  $\Sigma_{g,r}^s$ is the group of isotopy classes of orientation-preserving self-diffeomorphisms of $\Sigma_{g,r}^s$, which are compactly supported in the interior of $\Sigma_{g,r}^s$, and fixing $r$ marked points and the points on the boundary. For simplicity, we write $\Sigma_{g,r} = \Sigma_{g,r}^0$, $\Sigma_{g}^s = \Sigma_{g,0}^s$ and $\Sigma_{g} = \Sigma_{g,0}^0$. We also use the similar simplified notation for the corresponding mapping class groups. It is well-known that $\Gamma_{g,m}^r$ is generated by positive (right-handed) Dehn twists along non-separating curves. 

For a smooth surface bundle $f\colon E\to \Sigma$ with fibers $\Sigma_g^s$, the {\em monodromy representation} of $f$ is defined to be the map $\Psi \colon \pi _1(\Sigma)\to \Gamma _g^s$  relative to a fixed identification $\varphi$ of $F$ with the fiber over the base point of $\Sigma$: For each loop $\gamma \colon I\to \Sigma$ the bundle $f_\gamma \colon \gamma^* (E)\to I$ is canonically trivial, inducing a diffeomorphism $f_\gamma^{-1}(0)\to f_\gamma^{-1}(1)$ up to isotopy. Using $\varphi$ to identify $f_\gamma^{-1}(0)$ and $f_\gamma^{-1}(1)$ with $F$, we get the element $\Psi (\gamma)\in \Gamma _g$. Changing the identification $\varphi$ changes $\Psi$ by a conjugation with an element of $\Gamma_g$. We will use the functional notation for the mapping class group: i.e. for $f_1,f_2\in\Gamma_g$, the product $ f_1f_2$ means that we first apply $f_2$ and then $f_1$ --- thus the map $\Psi \colon \pi _1(\Sigma)\to \Gamma _g$ is an anti-homomorphism.

A genus--$g$ Lefschetz fibration $f \colon X\to\Sigma$ with a regular fiber $F \cong \Sigma_g$ can be defined combinatorially using the \emph{monodromy representation} $\Psi\colon \pi_1(B \setminus f(C)) \to \Gamma_g^r$, which determines $f$ up to isomorphism (and $X$ up to diffeomorphism), provided $g \geq 2$. (This is due to the fact that for $g\geq 2$ the space of self-diffeomorphisms of $F$ isotopic to the identity is contractible.) Importantly, isotopy type of a surface bundle over $S^1$ with fiber $F$ is determined by the return map of a flow transverse to the fibers, which can be identified with an element $\mu \in \Gamma_g$, called \dfn{monodromy} of this fibration over $S^1$.

It turns out that the monodromy of a Lefschetz fibration $f\colon X\to D^2$ over the disk with a single
critical point is a right Dehn twist\index{Dehn twist} along the vanishing cycle creating the singular
fiber. Therefore, the monodromy of a Lefschetz fibration $f\colon X \to \Sigma _h$ with $n$ critical points is given by a factorization of the identity element $1\in \Gamma _g$ as  
\begin{equation} \label{monodromyfactorization}
1=\prod _{i=1}^n t_{v_i}\prod _{j=1}^h [\alpha_j, \beta_j] \, ,
\end{equation}
where $v_i$ are the vanishing cycles of the singular fibers and $t_{v_i}$ is the positive Dehn twist about $v_i$. This factorization of the identity is called the \emph{monodromy factorization}. Here the mapping classes $a_i$ and $b_i$ specify the monodromies along a free generating system $\langle \alpha_1, \beta_1, \ldots , \alpha_h, \beta_h \rangle$ of $\pi_1 (\Sigma_h^1)$ such that $\prod_{i=1}^h [\alpha_i, \beta_i]$ is parallel to the boundary component of $\Sigma_h^1$. In particular, when there are no $t_{v_i}$ in the factorization, this prescribes a surface bundle. Conversely, a word
\[ 
w=\prod _{i=1}^n t_{v_i}\prod _{j=1}^h [\alpha_j, \beta_j] \]
prescribes a Lefschetz fibration over $\Sigma_h^1$, and if $w=1$ in $\Gamma_g^s$ we get a Lefschetz fibration $X\to \Sigma_h$.

For a Lefschetz fibration $f: X\to \Sigma$, a map $\sigma  \colon \Sigma  \to X$ is called a \emph{section} if $f \circ \sigma = id_{\Sigma}$. Suppose that a fibration $f\colon X\to \Sigma$ admits a section $\sigma$. Set $S = \sigma(\Sigma) \subset X$. This section $S$ provides a lift of the representation $\Psi: \pi_1 (\Sigma \setminus f(C)) \to \Gamma_g$ to the mapping class group $\Gamma _{g,1}$. One can then fix a disk neighborhood of this section preserved under the monodromy, and get a lift to $\Gamma_g^1$. Conversely, every such representation with a lift determines a fibration with a section: Gluing a disk with a marked point to a surface with one boundary component along the boundary and by extending self-diffeomorphisms of the surface by the identity on the disk, we obtain a surjective homomorphism $\Gamma ^1_g\to \Gamma _{g,1}$, whose kernel is freely generated by the right Dehn twist $t_\delta$ along a simple closed curve $\delta$ parallel to the boundary. If the factorization
 \[
 1=\prod _i t_{v_i} \prod _j [\alpha_j, \beta_j]
 \]
lifts from $\Gamma _g $ to a similar factorization in $\Gamma _{g,1}$, then the corresponding fibration has a section. Moreover, if we lift this product to $\Gamma _g^1$ we get
  \[
  t_\delta ^m=\prod _i t_{v'_i} \prod _j [\alpha'_j, \beta'_j]
   \]
for some $m$. Here, $t_{v'_i}$ is a Dehn twist mapped to $t_{v_i}$ under $\Gamma _g^1 \to \Gamma _g$. Similarly, $\alpha_j'$ and $\beta'_j$ are mapped to $\alpha_j$ and $\beta_j$, respectively. An elementary observation is that the power $m$ of $t_{\delta}$ in the above factorization in $\Gamma _g^1$ is the negative of the self-intersection number of the section $S$ that we obtain. 

These observations generalize in a straightforward fashion to the case when we have $r$ \emph{disjoint} sections $S_1, \ldots, S_r$, corresponding to $r$ marked points captured in the mapping class group $\Gamma_{g,r}^s$.

% ================================================================================================================
\subsection{ Open book decompositions} \
% ================================================================================================================

An \dfn{open book decomposition} $\mathcal{B}$ of a $3$--manifold $Y$ is a pair $(K,f)$ where $L$ is an oriented link in $Y$, called the \dfn{binding}, and $f\co Y \setminus K \to S^1$ is a fibration such that  $f^{-1}(t)$ is the interior of a compact oriented surface $F_t \subset Y$ and $\partial F_t=K$ for all $t \in S^1$. The surface $F=F_t$, for any $t$, is called the \dfn{page} of the open book. The \dfn{monodromy} of an open book is given by the return map of a flow transverse to the pages and meridional near the binding, which is an element $\mu \in \Gamma_{g, m}$, where $g$ is the genus of the page $F$, and $m$ is the number of components of $K =\partial F$. Equivalently, and more fitting with our later definitions, we can think of an open book decomposition as a decomposition of $Y = \mathcal{P} \cup \mathcal{S}$, where $\mathcal{P}$ is the fiber bundle $f:\mathcal{P} \to S^1$ (as before) with compact fibers, $\mathcal{S}$ is a union of solid tori $S^1 \x \{D^2_1, \dots, D^2_m\}$ (the neighborhoods of the binding components $K$) and each meridian disk $p \x D^2$ intersects the boundary of the fibers of $f$ in a single point. (Notice that up to isotopy, $D^2$ is determined by the topology of $\mathcal{S}$.)

Suppose we have a Lefschetz fibration $f\co X \to D^2$ with bounded regular fiber $F$, and let $p$ be a regular value in the interior of the base $D^2$. Composing $f$ with the radial projection $D^2 \setminus \{p\} \to \partial D^2$ we obtain an open book decomposition on $\partial X$ with binding $\partial f^{-1}(p)$. Identifying $f^{-1}(p) \cong F$,
we can write 
\[\partial X=(\partial F\times D^2)\cup f^{-1}(\partial D^2) \, .\]
Thus we view $\partial F\times D^2$ as the tubular neighborhood of the binding $K=\partial f^{-1}(p)$, and the fibers over $\partial D^2$ as its \dfn{truncated pages}. The monodromy of this open book is prescribed by that of the fibration. In this case, we say that the open book $(K, f|_{\partial X \setminus K})$ is \dfn{filled by}, or \dfn{induced by}, the Lefschetz fibration $(X,f)$. Any open book whose monodromy can be written as a product of positive Dehn twists can be filled by a Lefschetz fibration over the $2$-disk.

We can think of the second definition of an open book in this language as well. As a Lefschetz fibration, the boundary of $X$ inherits a K\"unneth-like decomposition consisting of vertical and horizontal boundaries (as viewed by $f$). In that case the fibered region $\mathcal{P}$ is the vertical boundary of $f$, $f^{-1} (\partial D^2)$, and $\mathcal{S}$ is the horizontal boundary, which is the (trivial) bundle of boundary circles $\partial F_t$ over $D^2$. As a bundle, we think of this as $f|_{\partial F}$. Each component of $\mathcal{S}$ is the topologically $S^1 \x D^2$ and there is a unique isotopy class of section which trivializes the bundle. 

% ================================================================================================================
\subsection{ Contact structures and compatibility } \
% ================================================================================================================

A $1$--form $\alpha \in \Omega^1(Y)$ on a $(2n{-}1)$--dimensional oriented manifold $Y$ is called a \dfn{contact form} if it satisfies $\alpha \wedge (d\alpha)^{n-1} \neq 0$. A \dfn{co-oriented contact structure} on $Y$ is then a hyperplane field $\xi$ which is globally written as the kernel of a contact $1$--form $\alpha$. In dimension three, this is equivalent to asking $d\alpha$ to be nondegenerate on the plane field $\xi$.

A contact structure $\xi$ on a $3$--manifold $Y$ is said to be \dfn{supported by an open book} $\mathcal{B}=(K,f)$ if $\xi$ is isotopic to a contact structure given by a $1$--form $\alpha$ satisfying $\alpha>0$ on positively oriented tangents to $K$ and $d\alpha$ is a positive volume form on every page. When this holds, we say that the open book $\mathcal{B}$ is \dfn{compatible with the contact structure} $\xi$ on $Y$. It is a classical result of Thurston and Winkelnkemper \cite{TW} that any open book admits such a contact structure (where the ``compatibility'' definition is due to Giroux).

Considering contact $3$--manifolds as boundaries of certain $4$--manifolds together with various compatibility conditions has been an active research topic in low dimensional topology. From the contact topology point of view,
it is the study of different types of \dfn{fillings} of a fixed contact manifold. In dimension four, there are essentially two considerations. Let $(X,\omega)$ be a symplectic $4$-manifold with cooriented nonempty boundary $Y=\partial X$. If there exists a \emph{Liouville vector field} $\nu$ defined on a neighborhood of $\partial X$ pointing out along $\partial X$, then we obtain a positive contact structure $\xi$ on $\partial X$, which can be written as the kernel of contact $1$--form $\alpha = \iota_{\nu} \omega|_{\partial X}$. When this holds, we say $(Y, \xi)$ is the \dfn{$\omega$--convex boundary} or \dfn{strongly convex boundary} of $(X, \omega)$. (When $\nu$ points inside, we say $(Y, \xi)$ is the \dfn{$\omega$--concave boundary of $(X, \omega)$}.)

Now if $(X, J)$ is almost-complex, then the complex tangencies on $Y=\partial X$ give a uniquely defined oriented hyperplane field. It follows that there is a $1$--form $\alpha$ on $Y$ such that $\xi = Ker \alpha$. We define the \dfn{Levi form} on $Y$ as $d\alpha|_{\xi}(\cdot, J \cdot)$. If this form is positive definite then $(Y, \xi)$ is said to be \dfn{strictly $J$--convex boundary of $(X,J)$}, and if it is $J$--convex for an unspecified $J$ (for instance when $J$ is tamed by a given symplectic form), we say $(Y, \xi)$ is \dfn{strictly pseudoconvex boundary}. If $(X, \omega, J)$ is an almost-K\"ahler manifold, i.e. a manifold equipped with a symplectic form $\omega$ and a compatible almost-complex structure $J$, then it can be shown that strict pseudoconvexity of the boundary is equivalent to the condition
that ${\omega|}_{\xi } > 0$ . 

For detailed and comparative discussions of these concepts, as well as proofs of some facts mentioned in the next subsection, the reader can turn to \cite{ElGr} and \cite{E}. For further basic notions from contact topology of $3$--manifolds such as Legendrian knots, Thurston--Bennequin framing, which appears below, the text of Ozbagci--Stipsicz \cite{OS1} or that of Geiges \cite{Geiges} would be valuable sources.

% ================================================================================================================
\subsection{ Stein manifolds} \
% ================================================================================================================

A smooth function $\psi\co X \to \R$ on a complex manifold $X$ of real dimension $2n$ is called \dfn{strictly plurisubharmonic} if $\psi$ is strictly subharmonic on every holomorphic curve in $X$. We call
a complex manifold $X$ \dfn{Stein}, if it admits a proper strictly plurisubharmonic function $\psi\co X \to [0, \infty)$ (\,after Grauert \cite{Gra}). Thus a compact manifold $X$ with boundary which is equipped with a complex structure in its interior is called \dfn{compact Stein} if it admits a proper strictly plurisubharmonic function which is constant
on the boundary.

Given a function $\psi\co X \to \R$ on a Stein manifold, we can define a $2$--form $\omega_{\psi} = -dJ^* d\psi$. It turns out that $\psi$ is a strictly plurisubharmonic function if and only if the symmetric form $g_{\psi}(\cdot, \cdot) = \omega_{\psi}(\cdot, J \cdot)$ is positive definite. So every Stein manifold $X$ admits a K\"ahler structure $\omega_{\psi}$, for any strictly plurisubharmonic function $\psi\co X \to [0, \infty)$. It is easy to see that the restriction of $\omega_{\psi}$ to each level set $\psi^{-1}(t)$ gives a Levi form on $\psi^{-1}(t)$, implying that all nonsingular level sets of $\psi$ are strictly pseudoconvex hypersurfaces. Thus one can equivalently
call a Stein manifold a \dfn{strictly pseudoconvex manifold}. Moreover, it was observed in \cite{ElGr} that the gradient vector field of $\psi$ defines a (global) Liouville vector field $\nu=\nabla_{\psi}$, making all nonsingular level sets $\omega_{\psi}$--convex. Hence, Stein manifolds exhibit strongest filling properties for a contact manifold which can
be realized as their boundary. Given contact $3$-manifold $(Y,\xi)$, we will call  Stein surface $(X,J)$ a \dfn{Stein filling} of $(Y, \xi)$ if $\partial X = Y$ and $J|_Y$ induces the contact structure $\xi$. 

In this article, we are mainly interested in compact Stein surfaces with convex boundaries, up to diffeomorphisms. A topologist's characterization of these manifolds in terms of Weinstein structures (cf. \cite{Weinstein}, \cite{CiEl}) follows from the work of Eliashberg and Gompf:

\begin{theorem} [Eliashberg \cite{El1}, Gompf \cite{Go2}]
\label{Eliashberg}
A smooth oriented compact $4$--manifold with boundary is a Stein surface, up to orientation preserving diffeomorphisms, if and only if it has a handle decomposition $X_0 \cup h_1 \cup \ldots \cup h_m$, where  $X_0$ consists of \,$0$-- and \,$1$--handles and each $h_i$, $1\leq i \leq m$, is a $2$--handle attached to 
\[ X_i= X_0 \cup h_1 \cup \ldots \cup h_i \]
along a Legendrian circle $L_i$ with framing $tb(L_i) -1 $.
\end{theorem}

\begin{theorem} [Loi--Piergallini \cite{LP}, also see Akbulut--Ozbagci \cite{AO}] \label{PALF}
An oriented compact $4$--manifold with boundary is a Stein surface, up to orientation preserving diffeomorphisms, if and only if it admits an allowable Lefschetz fibration over the $2$-disk, a.k.a ``PALF''. Moreover, any two allowable Lefschetz fibrations over the $2$-disk filling the same open book carry Stein structures which fill the same contact structure (induced by the open book). 
\end{theorem}

% ================================================================================================================
\subsection{Convex surfaces} \
% ================================================================================================================

In this article, we will make extensive use of convex surface theory, which we review briefly here. For details and proofs, see \cite{Ho}. A surface $S$ in a contact $3$-manifold $Y$ is \emph{convex with Legendrian boundary} if any boundary component of $S$ is tangent to the contact planes and there is a vector field $X$ defined in a neighborhood of $S$ that is positively transverse to $S$ and which preserves the contact planes. In that case, we assume that $X$ is transverse to $\xi$ and let $S_+$ denote the set of points for which $X$ is positively transverse to $\xi$, $S_-$ the set of points where $X$ is negatively transverse to $\xi$, and $\Gamma$ the set where $X$ is tangent to $\xi$. $\Gamma$ is then a collection of properly embedded, simple closed curves which separate $S_+$ and $S_-$ called the \emph{dividing set}. 

\begin{theorem}[Giroux \cite{Gi1}, Honda \cite{Ho}] \label{thm:giroux flexibility} For a convex surface $S$ with Legendrian boundary, the subsets $S_+$ and $S_-$ are embedded submanifolds whose boundary constitute a collection of properly embedded circles $\Gamma$. Further, the isotopy class of $\xi$ in a neighborhood of $S$ is determined by $\Gamma$.
\end{theorem}

\noindent The \emph{standard convex $S^2$} has a single circle as its dividing set. The \emph{standard $3$-ball} is the contact manifold which is tight on $B^3$ and with boundary the standard convex $S^2$. A \emph{bypass} is a convex bigon with Legendrian boundary and whose dividing set consists of a single arc with both boundary points on the same boundary arc.

A contact $3$-manifold $Y$ admits a \emph{decomposing disk} if there is a proper, non-boundary parallel convex disk $D$ with Legendrian boundary whose dividing set consists of a single arc. We say $Y$ is \emph{disk decomposable} if there is a collection of disjoint decomposing disks so that cutting and rounding gives a collection of standard contact $3$-balls. A \emph{product contact manifold} is a contact manifold, diffeomorphic to $F \times I$ for some compact, convex surface $F$ with Legendrian boundary. The notions of a disk decomposable and product contact manifold are equivalent up to smoothing the boundary.

An \emph{$S^1$-invariant contact structure} is a contact structure on a surface bundle over $S^1$ with convex torus (or empty) boundary, whose fibers are all convex surfaces. Equivalently, an $S^1$-invariant contact structure is made by taking a product contact manifold and gluing the top to the bottom by a diffeomorphism preserving the dividing set.

\vspace{0.1in}
%\newpage
% ================================================================================================================
\section{Contact structures on spinal open books and Stein structures on allowable Lefschetz fibrations over arbitrary surfaces} \label{Spinal}
% ================================================================================================================
% ================================================================================================================

% ================================================================================================================
\subsection{Spinal open books} \
% ================================================================================================================

The notion of a spinal open book was introduced in \cite{LVW} and used to classify fillings of certain contact manifolds. It is (roughly speaking) the right kind of structure to study contact structures arising as the boundaries of Lefschetz fibrations over non-disk bases. We give a set of proofs and constructions based on convex surface theory in this section. In the appendix, Lisi and Wendl give what should be considered the standard characterization of compatibility, existence and uniqueness of contact structures in terms of Reeb fields and Giroux forms. The following definitions are equivalent but have been altered to accommodate the spinal tap construction of Section~\ref{spinal tap}. A \emph{spinal open book decomposition} $\mathcal{B}$ of a $3$-manifold $Y$ is a decomposition of $Y$ into regions $\mathcal{P} \cup_{T} \mathcal{S}$, where 
\begin{itemize}
\item $\mathcal{P}$ is a compact, embedded, codimension-0 submanifold with torus boundary components, equipped with the structure of a fiber bundle $\whf \hookrightarrow \mathcal{P} \xrightarrow{\pi_{\mathcal{P}}} S^1$ for some possibly disconnected surface $\whf$ with boundary.

\item $\mathcal{S}$ is similarly a compact, embedded, codimension-0 submanifold with torus boundary (the same boundary as $\partial \mathcal{P}$), equipped with a the structure of a circle bundle $S^1 \hookrightarrow \mathcal{S} \xrightarrow{\pi_{\mathcal{S}}} \whs$, over a disjoint union of surfaces with non-empty boundaries.

\item The (oriented) boundary components of a fiber $F$ in $\mathcal{P}$ are $S^1$ fibers of $\pi_{\mathcal{S}}$ (equipped with the same orientation). 
\end{itemize}

\noindent We call the fibered region $\mathcal{P}$ the \emph{paper}, the fibers $F$ the \emph{pages}. The product region $\mathcal{S}$ we call the \emph{spine} and for any section of $\mathcal{S}$, we call a connected component, $\Sigma$, a \emph{vertebra}. The tori boundary $T$ between the two we call \emph{interface tori}. 

For the purposes of this paper, all spinal open books will be \emph{symmetric, uniform and simple} (in the terminology of \cite{LVW}). By this we mean every component of $\whf$ is homeomorphic, every component of $\whs$ is homeomorphic, and every component of $\mathcal{S}$ is adjacent to every component of $\mathcal{P}$ along a single interface torus.

An \textit{abstract spinal open book} is a $5$-tuple $(Y, \whf, \widehat{\phi}, \whs, G)$ where:

\enlargethispage{0.5in}
\begin{itemize}
\item  $Y$ is a closed $3$-manifold
\item $\whf$ is a disjoint union of surfaces with non-empty boundaries: \\
 $\whf = F_1 \cup \cdots \cup F_n$
\item $\widehat{\phi}$ is an orientation preserving self-diffeomorphism of $\whf$ fixing the boundary pointwise 
\item $\whs$ is a disjoint union of surfaces with non-empty boundaries: \\
$\whs = \Sigma_1 \cup \cdots \cup \Sigma_m$
\item ${G}$ is a bijection $G:|\partial \whf| \cong |\partial \whs|$
\end{itemize}

\noindent To construct an isomorphism class of embedded spinal open books from this, we form the surface bundle over $S^1$ with fiber $\whf$ and monodromy $\widehat{\phi}$, and the trivial bundle $S^1 \x \whs$. We glue the resulting boundaries together using $G$ to identify components and so that the oriented boundary of a fiber $\whf$ is a collection of $S^1$ fibers in $S^1 \times \whs$. Note that the monodromy condition on an abstract open book means we are automatically constructing a \emph{framed} spinal open book decomposition: the boundary of $\whs$ is the orbit of a point (one per boundary component of $\whf$) under the self-diffeomorphism $\widehat{\phi}$. 

We say a spinal open book decomposition $\mathcal{B}$ \emph{carries} (or \emph{admits} or is \emph{compatible with}) a contact structure $\xi$ if there exists a contact form $\alpha$ for $\xi$ whose Reeb vector field is positively transverse to both the pages of $\mathcal{P}$ and the sections of the spine $\mathcal{S}$. In that case, we also say that $\xi$ is \emph{supported by} $\mathcal{B}$. (Note, for every spinal open book, there we can isotope $\xi$ so there is a contact form for $\xi$ whose Reeb field is positively tangent to all $S^1$ fibers, cf. \cite{LVW}.)  The triple $(Y, \mathcal{B}, \xi)$ will then denote a spinal open book $\mathcal{B}$ and a contact structure $\xi$ on the closed $3$-manifold $Y$ compatible with each other. 

\noindent The reader should compare this to the Thurston--Winkelnkemper construction for regular open book decompositions \cite{TW}.

\vspace{0.1in}
% ================================================================================================================
\subsection{Framed spinal open books} \
% ================================================================================================================

For our purposes an equivalent definition of compatibility between spinal open books and contact structures will be useful, for which we first introduce the following: A \emph{framed spinal open book decomposition} is a spinal open book decomposition along with a chosen section of the spine. Specifically, it is a spinal open book decomposition along with an identification of $\mathcal{S}$ with $\whs \times S^1$.

%pdftex or latex - change extension
\begin{figure}
\labellist
\pinlabel $F$ at 237 387
\pinlabel $\Sigma$ at 321 401
\endlabellist

\includegraphics[width = 2.4in]{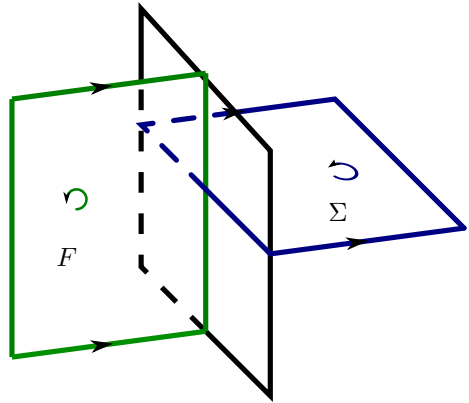}
		\caption{Orientations of the fiber and vertebrae at an interface torus ${T}$.}
	\label{fig:interfacetorus}
\end{figure}

%We point out that the framing referred to in the previous definition is the identification of $\mathcal{S}$ with the product manifold $\widehat{\Sigma} \times S^1$, or equivalently, the choice of a section of the $S^1$-bundle. If $\widehat{\Sigma}$ is not a disjoint union of disks, there are infinitely many ways of writing the total space as a product. We will show in Lemma \ref{lm:framing change} that changing the framing does not change the associated contact structure. Having a framing, though, will be useful in later discussions.

A \emph{framed spinal open book decomposition} carries (or is compatible with) a contact structure $\xi$ if the following conditions are satisfied:

\begin{itemize}
\item the interface tori are convex with dividing set two parallel curves of negative slope (i.e., in the $(\partial \Sigma, \partial F)$-basis, the dividing set is of the form $\pm(-p,q)$ for $p, q>0$).
\item on each component of the paper $\mathcal{P}$, we can isotope a page $F$ convex with Legendrian boundary on $\mathcal{T}$ and with a dividing set consisting of boundary parallel arcs so that then negative regions $F_-$ are boundary parallel bigons, and so that after cutting and rounding on $F$, $\mathcal{P}$ is a product contact manifold. (Equivalently, we ask that the complement of $F$ in $\mathcal{P}$ be \emph{disk decomposable}.) 
\item on each component of the spine $\mathcal{S}$, we can make a vertebra $\Sigma$ convex with Legendrian boundary on $\mathcal{T}$ and with dividing set parallel to each component of the boundary of $\Sigma$, so that after cutting and rounding on $\Sigma$, $\mathcal{S}$ is a product contact manifold.
\end{itemize}

\noindent Intuitively, one should think of the contact structure associated to a framed spinal open book as being a deformation of the tangent planes to the fibers and vertebrae, and rotating a quarter turn between them in a small neighborhood of the interface torus, just as the contact structure we associate to a standard open book is a deformation of the tangent planes to the fibers and to a small disk neighborhood of the binding, with a quarter-turn rotation in between.

\begin{remark} The requirement that the slope of the interface torus be negative is necessary for a good definition of compatibility. Any contact structure given above with a positively sloped dividing set on $T$ is overtwisted with an overtwisted disk located in a small neighborhood of the interface torus. (See Figure  \ref{fig:bypass}.) In addition, notice that Lemma \ref{lm:framing change} (plus Proposition \ref{prop:uniqueframed}) implies that \emph{any} negative slope can be realized on the interface torus. \end{remark}

\begin{proposition} \label{prop:uniqueframed}
Every framed spinal open book decomposition admits a unique isotopy class of a compatible contact structure.
\end{proposition}

\begin{proof} From the description of compatibility via convex surfaces, given the dividing set on the interface tori $T$, such a contact structure both exists and is unique up to isotopy. To see this, we can make the given interface tori convex. After thickening, we can make the fiber and the vertebra simultaneously convex with Legendrian boundary and with the specified dividing set. A neighborhood of this union has a unique contact structure and complement is disk decomposable. 

To show that the definition is well-defined, though, we need to see that it is independent of the slope of the dividing set on the interface torus. This is guaranteed by the orientations chosen and described in Figure \ref{fig:interfacetorus}. In particular, we could have chosen the slope to be $-1$. Here are the details: 

By switching to the a framed spinal open book, we can make both $\mathcal{S}$ and $\mathcal{P}$ are contact bundles with convex fibers. Let $\whs$ be the fiber of $\mathcal{S}$ and let $\Sigma$ be a connected component of $\whs$. Similarly, $\whf$ is the fiber $\mathcal{P}$ and $F$ is a connected component of $\whf$. Since the interface torus has dividing sets of slope $(-p,q)$, the fiber $\Sigma$ has $q$ components in the dividing set parallel to $T$, and $F$ has $p$ components.  We want to show that we can decrease each of $q$ and $p$ to 1 while still keeping $\mathcal{S}$ and $\mathcal{P}$ contact bundles whose convex fibers have boundary parallel dividing sets. To do this, observe that each dividing curve on $F$ determines a bypass for $\mathcal{S}$ and each dividing curve for $\Sigma$ determines a bypass for $\mathcal{P}$. Forgetting the contact structure on the complement of $T$, if we attach $p-1$ bypasses from $F$ and $q-1$ bypasses from $\Sigma$, the resulting torus has slope $-1$. It suffices to show, then, that after sliding along one of these bypasses, the resulting spinal open book remains compatible. Thus the next lemma completes the proof.
\end{proof}

\begin{lemma} \label{lm:slide} If we slide $T$ over a bypass from $F$, the resulting spinal open book remains compatible with the contact structure, and similarly for a bypass from $\Sigma$. \end{lemma}

\begin{proof} We know the contact structure on $\mathcal{P}$ is tight and the complement of $F$ is disk decomposable. Let $\mathcal{P}'$ be the result of cutting out the bypass layer from $\mathcal{P}$ and $F'$ be the subsurface of $\mathcal{P}'$ consisting of $F$ with the bypass removed. Then the contact structure on $\mathcal{P'}$ is tight and the complement of $F'$ is disk decomposable.  Thus $\xi$ is compatible along $\mathcal{P'}$.

\begin{figure} 
\labellist
\pinlabel $\Sigma$ at 200 574 
\endlabellist
\includegraphics[width = 2in]{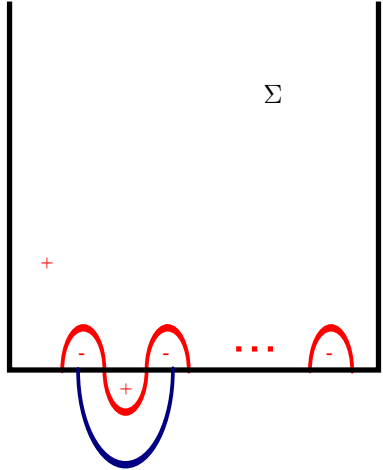}
\caption{ \label{fig:bypass} Attaching a bypass from the pages to the spine and vice versa.}
\end{figure}

Now suppose we attach this bypass to $\mathcal{S}$. Coming from $F$, this bypass is being attached along a vertical Legendrian arc straddling three adjacent arcs of the dividing set. The bypass arc can be slid down the $T$ so it is parallel with the given section $\Sigma$ of the spine. Because the dividing set on $T$ was chosen with the appropriate slope, attaching this bypass merges two adjacent disks in $\Sigma_-$ (as opposed to capping a single disk). The dividing sets of $\Sigma$ and this added bypass are shown in Figure \ref{fig:bypass} As before, after cutting along $\Sigma$, now extended by the bypass, the result is disk decomposable. 
\end{proof}

One of the nice corollaries of the Giroux correspondence is that if one knows a contact structure is tight, then it is determined by a single page of a compatible open book. From the proof of invariance above, the contact structure on a framed spinal open book is determined up to isotopy by a single (possibly disconnected) page and a single (possibly disconnected) section of the spine, and the interface tori, as topological submanifolds. 

The next lemma shows how the framings on a given spinal open book relate to each other, indicating how we get an equivalent definition of a compatible contact structure on an unframed \textit{spinal open book}, up to isotopy. 

\begin{lemma}\label{lm:framing change} Let $\mathcal{B}$ and $\mathcal{B}'$ be two different \emph{framed} spinal open book decompositions which represent the same spinal open book. In particular, $\mathcal{B}$ and $\mathcal{B}'$ correspond to two different choices of sections of $\mathcal{S}$. Then the two contact structures $\xi$ and $\xi'$ carried by $\mathcal{B}$ and $\mathcal{B}'$  are isotopic. \end{lemma}

\begin{figure} 
\labellist
\pinlabel $\pi(\Sigma)$ at 277 545
\endlabellist
\includegraphics[width = 2.5in]{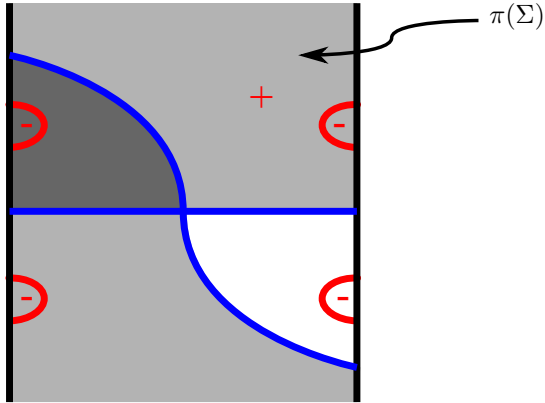}
\caption{\label{fig:section} The projection of the graph of a convex section of the spine after being spun along an arc.}
\end{figure}

\begin{proof} Changing a section of $\mathcal{S}$ is equivalent to choosing a map from $\whs$ to $S^1$. Such a map is determined by the degree on a basis of $\pi_1(\whs)$. Changing the degree by 1 on a single generator is equivalent to taking a properly embedded dual arc $a$ in a component $\whs$ and ``spinning" the section $\Sigma$ along that dual arc. Equivalently, if we take the product annulus sitting over this arc, we can form the oriented resolution of the section with this annulus. In any case, it is enough to show that by adding a single annulus to $\Sigma$ in $\mathcal{B}$, we get a new framed spinal open book compatible with the same (isotopy class of) contact structure. Since we know that the contact structure on $\mathcal{S}$ is tight, it's enough to find a convex representative of the new section with boundary parallel dividing curves.

For ease, assume that we have arranged the contact structure so that the interface tori at the boundaries of $\mathcal{S}$ touching $\partial a$ have dividing sets of slope $\pm(-1,2)$ in the $(\partial \Sigma, \partial F)$ basis. Choose a representative of $a$ on $\Sigma$ which is disjoint from the dividing set. If we slide $a$ in the vertical direction, keeping it disjoint from all the fibers, it will return to $\Sigma$ having moved to the right by jumping over one disk component of $\Sigma_-$ on each boundary. We can spin $\Sigma$ in the vertical direction in a small neighborhood of $a$, keeping it convex. In order to glue to get a closed surface, though, we need to remove a small triangle of $\Sigma$ on one boundary, and wrap by an additional triangle on the other as shown in Figure \ref{fig:section}. (This has the effect of removing one bigon component of $\Sigma_-$ on one boundary and adding bigon component on the other boundary.) This gives a new section with boundary parallel dividing curves, as required.
\end{proof}

Combining Lemma \ref{lm:framing change} and Proposition \ref{prop:uniqueframed} gives a new proof of the following result from \cite{LVW}:

\begin{proposition}\label{prop:welldefined} Every spinal open book decomposition is compatible with a unique isotopy class of contact structure. 
\end{proposition}

%\vspace{0.2in}
%\newpage
%
%

\vspace{0.1in}
% ================================================================================================================
\subsection{Spinal tap on a spinal open book} \label{spinal tap}
% ================================================================================================================

We will now define an operation on embedded spinal open books, which comes with a natural Stein cobordism, as we will discuss shortly. This operation (in both directions) has been studied already by Baldwin \cite{Baldwin}. Avdek \cite{Avdek} gives the inverse operation. For ease, we restrict to symmetric, uniform, simple open books, though the operation works in much more generality. 

As a motivation, we outline the plan to prove Theorem~\ref{thm2}. Suppose we start with a Lefschetz fibration $(X,f)$ with a connected, bounded, non-disk base. To construct a Stein structure on $(X,f)$ we could cut $f$ along an arc decomposition of $\Sigma$ to get a Lefschetz fibration over the disk -- this is known to admit a Stein structure. We then extend the Stein structure back across the 1-handles of $\Sigma$, showing that at every stage we get a spinal open book at the boundary of the Lefschetz fibration, which is compatible with the contact structure at the boundary of the Stein structure. 

We start with a topological cut-and-paste operation on spinal open books.

	For an embedded spinal open book $\mathcal{B} = \mathcal{P} \cup \mathcal{S}$, let $S$ be a surface consisting of two components $F_1$ and $F_2$ of a fiber in $\mathcal{P}$ along with some annuli of made up of the $S^1$ fibers in $\mathcal{S}$ connecting them, one annulus in each component of $\mathcal{S}$. Since the orientations of $F_1$ and $F_2$ don't agree along the annuli we orient $S$ as $F_1 \cup - F_2$. Call such a surface a \emph{spinal tap surface}. A \emph{spinal tap} along $S$ is the following operation:

\begin{itemize} 
\item Cut $\mathcal{B}$ along $S$. The resulting manifold has two boundary components $S_+ = F_{1+} \cup - F_{2+}$ and $S_- = F_{1+} \cup - F_{2+}$. On 
\item Fold $S_+$ by gluing $F_{1+}$ to $F_{2+}$ by a diffeomorphism $h: F_1 \to F_2$.
\item Fold $S_-$ by gluing $F_{2-}$ to $F_{1-}$ by the inverse diffeomorphism $h^{-1}: F_2 \to F_1$.
\end{itemize}
  
The resulting open book $\mathcal{B}' = \mathcal{P}' \cup \mathcal{S}' $ has 
\begin{itemize}
\item the new spine $\mathcal{S}'$ is $\mathcal{S}$, cut along the connecting annuli of $S$
\item $\mathcal{P}'$ is the bundle made by cutting $\mathcal{P}$ out $F_1 \x [0,1]$ and $F_2 \x [0,1]$ and identifying $F_1 \x \{0\} $ and $F_2 \x \{1\}$ by $h$ and $F_2 \x \{0\} $ and $F_1 \x \{1\}$ by $h^{-1}$.
\end{itemize}

For our purposes it is helpful to see how to construct a spinal tap abstractly.

Let $\mathcal{B} = (Y, \whf, \widehat{\phi}, \whs, G)$ be an abstract spinal open book. We can form a new spinal open book $\mathcal{B}' = (Y', \whf', \widehat{\phi}', \whs', G')$ as follows: 
\begin{itemize}
\item pick a set of identifications $\hat{i}: \whs \to \base$ of each component of the spine with an abstract surface $\base$
\item choose a properly embedded arc $a$ in $\base$
\item we can isotope $\hat{i}$ and $a$ so that that for each component of $\hat{i}^{-1}(a)$, the fiber circles in the interface torus $T = \partial \mathcal{S}$ above $\partial a$ are the boundary circles of precisely two connected fiber components, $F_1$ and $F_2$ in $\mathcal{P}$
\item in particular, $\partial a$ sits on two components of $\partial \base$, each of which determines a component of $\mathcal{P}$. Let $F_1$ be a fiber in one such component of $\mathcal{P}$ and $F_2$ in the other. Then $F_1$ and $-F_2$ can be glued together along the annuli $\pi_{\mathcal{S}}^{-1}(\hat{i}^{-1} (a))$ to get a closed surface $S$ in $Y$. Call such a surface a \emph{spinal tap surface} of $\mathcal{B}$.\\
Notice that due to the orientation conventions, we need to reverse the orientation on one of these surfaces from the orientation coming from the bundle.

\item topologically, we can cut $Y$ along $S$ and glue in two copies of $F_1\x [0,1]$ to get a new spinal open book $\mathcal{B}_0$. In particular, cutting $Y$ along $S$ leaves two boundary components $S_+$ and $S_-$. Equipped with the orientation from $F_1$ and $S$, $S_- = F_1 \cup - F_2$ and $S_+ = - F_1 \cup F_2$. (With this choice of orientation, $S_-$ is oriented as the boundary of $Y - S$, while $S_+$ has the opposite orientation.) Choose an orientation preserving identification $h:F_1 \xrightarrow{\cong} F_2$ which preserves the identification of the boundaries in $S$ and glue in two copies of $F_1\x [0,1]$ so that 
$$\begin{array}{c} F_1 \x \{0\} \xrightarrow{id} F_1 \\ F_1 \x \{1\} \xrightarrow{h} F_2 \\
\end{array}$$
on $S_+$
and 
$$\begin{array}{c} F_1 \x \{0\} \xrightarrow{h} F_2 \\ F_1 \x \{1 \} \xrightarrow{id} F_1 \\
\end{array}$$
on $S_-$
\end{itemize}

The resulting spinal open book $\mathcal{B}' $ is made by removing an arc from each component of $\whs$. If this arc connects two different boundary components of $\base$, then we compose (or concatenate) the two fibered regions along the fibers $F_1$ and $F_2$. If the arc has both boundary points on the same boundary of $\base$, then we cut one fibered region along two different fibers and close them up so to form two different fiber bundles.
	
\begin{proposition}\label{prop:unfolding} Suppose $(Y, \mathcal{B}, \xi)$ is a contact spinal open book and suppose that $(Y', \mathcal{B}', \xi')$ is obtained from $\mathcal{B}$ by a spinal tap. Then there is a Stein cobordism from $(Y', \xi')$ to $(Y, \xi)$. Moreover, if $\mathcal{B}'$ is the boundary of a Lefschetz fibration $(X', f')$, then this Lefschetz fibration can be extended, along this Stein cobordism, to a Lefschetz fibration $(X, f)$ with boundary $\mathcal{B}$. 
\end{proposition} 	

The proof of this proposition is broken down into two parts. First we construct the desired Stein cobordism, verifying that the two contact structures at either end of the cobordism are compatible with the specified spinal open book. Then we show that this cobordism behaves nicely with respect to a Lefschetz fibration with boundary $\mathcal{B}'$.

We'll call a closed convex (or sutured) surface \emph{foldable} if it admits an orientation reversing self-diffeomorphism, fixing the dividing set (or sutures) and sending $R_+$ to $R_-$ (and vice versa). Equivalently, a foldable convex surface is one that arises as the boundary of an $I$-invariant neighborhood of a convex surface with boundary parallel dividing set (or as the boundary of a product contact or sutured manifold).

The following proposition shows that the there are nice surfaces in spinal open books (which generalize the idea of a union of a pair of fibers in an open book as in \cite{Torisu} for example), for which the \emph{spinal tap} operation on spinal open books described above is equivalent to the convex cut-and-paste operation of \emph{folding} along a convex surface.

\begin{proposition} Let $S = F_1 \cup F_2$ be a spinal tap surface in a spinal open book $\mathcal{B}$ and let $\mathcal{B}'$ be the result of a spinal tap along $S$. Denote by $\xi$ the contact structure supported by $\mathcal{B}$ and by $\xi'$ that of  $\mathcal{B}'$. Then we can make $S$ a foldable convex surface in $\xi$ so that $\xi'$ is the result of folding $\xi$ along $S$. 
\end{proposition}

\begin{proof} First we need to construct a suitable convex surface isotopic to the spinal tap surface $S$. Let $\Sigma$ be vertebra in a component of the spine and let $a$ be the isotopy class of the arc in $\Sigma$ used to construct $S$. Let $F_1$ and $F_2$ be two fibers of $\mathcal{P}$ with boundary isotopic to the circle in $\mathcal{S}$ above $\partial a$. We want to construct a nice convex representative of the annulus $A$ over $a$ so that the boundary circles are Legendrian isotopic in a product neighborhood of $T$ to the boundaries of $F_1$ and $F_2$. We can Legendrian realize $a$ on $\Sigma$ so that is misses the dividing set. After sliding around the $S^1$ direction, $a$ will jump some number of boundary parallel bigons of $\Sigma_-$ to the right on each boundary. (cf. Figure \ref{fig:section}.) We want to construct $A$ in two steps, first moving to the right along one boundary component,then moving to the right on the other. In each case, we sweep out a subsurface of $\Sigma$, which can be made convex with dividing set that inherited from $\Sigma$. However, the orientations on these subsurfaces do not agree, and so we need to ``fold'' the two surfaces to match them up as in Figure \ref{fig:folded annulus}. Doing so gives a nice, convex vertical annulus $A$ with dividing set shown in Figure \ref{fig:convex annulus}. Since the boundary circles of $A$ are formed from a vertical arc disjoint from the dividing set and a horizontal annulus contained in the boundary of $\Sigma$, they are isotopic within the product neighborhood of $T$ to the boundaries of the fibers. 

\begin{figure}
\labellist
\pinlabel $a$ at 82 24 
\pinlabel $a$ at 85 112
\endlabellist
\includegraphics[width=1.7in]{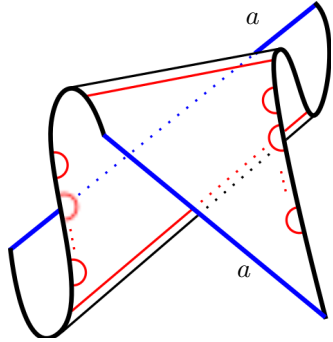}
\caption{\label{fig:folded annulus} A local projection of a vertical annulus being made convex.}
\end{figure}

\begin{figure}

\includegraphics[width=1.5in]{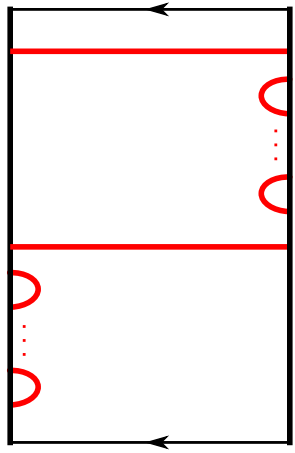}
\caption{\label{fig:convex annulus} The dividing set on the convex annulus $A$.}
\end{figure}

The new spine $\mathcal{S}'$ will be formed by cutting $\mathcal{S}$ along $A$. In particular, we want to add a neighborhood of $A$ to $\mathcal{P}$ and round the resulting boundary to get our new interface tori $T'$, with $\mathcal{S}'$ being the portion of $\mathcal{S}$ on the inside of $T'$ and taking $\Sigma \backslash a$, rounded, to be our new section $\Sigma'$. A schematic is shown in Figure \ref{fig:sigma view}. Since $\mathcal{S}'$ is tight (as a subset of $\mathcal{S}$), cutting $\mathcal{S}'$ along $\Sigma'$ and rounding gives a disk decomposable handlebody, and so $\xi_0$ and $\mathcal{S}'$ remain compatible.

\begin{figure}
\vspace{.1in}
\labellist
\pinlabel {\LARGE $\mathcal{P}$} at 129 468
\pinlabel {\LARGE $\mathcal{P}$} at 246 468
\pinlabel {\LARGE $\mathcal{S}$} at 177 471
\pinlabel $T$ at 159 486
\pinlabel $T$ at 223 486
\pinlabel $T'$ at 168 451
\pinlabel $T'$ at 168 427 
\pinlabel $F_1$ at 116 439
\pinlabel $F_2$ at 269 439
\pinlabel $a$ at 218 443
\pinlabel $\Sigma$ at 213 472
\endlabellist
\includegraphics[width=4in]{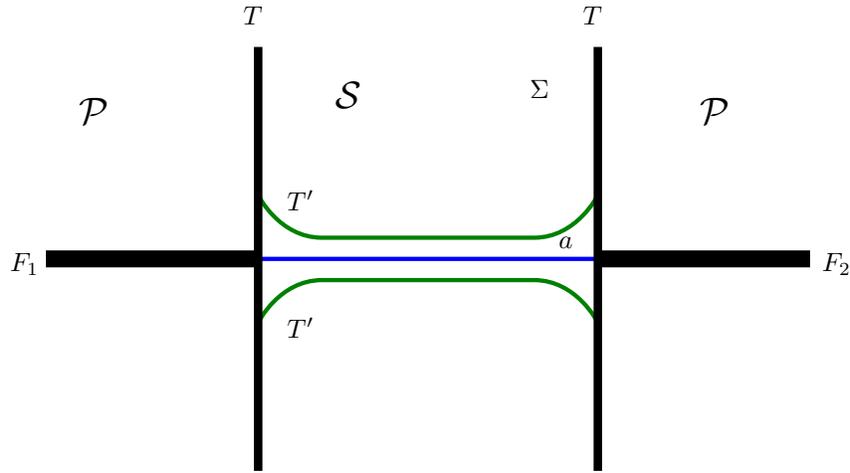}
\caption{\label{fig:sigma view} A look at the spine and paper projected to $\Sigma$. The new interface tori and spine are shown.}
\end{figure}

Now let's look at the complement of $\mathcal{S}'$, which is $\mathcal{P} \cup \nu(A)$ with the boundary rounded. To construct $\xi'$, we will cut along $S$ and fold the two boundaries back upon themselves. However, it will be easier to actually crimp the edges of $Y \backslash S$ and glue in a product contact manifold $F_1 \times [0,1]$, keeping $A$ and it's standard neighborhood isolated. A schematic of this is shown in Figure \ref{fig:crimping the paper}. In particular, the horizontal boundary components are the original fiber surfaces $F_1$ and $F_2$, and the vertical surface (as shown in Figure \ref{fig:crimping the paper}) is the convex structure we placed on the annulus $A$.

\begin{figure}
\noindent\makebox[\textwidth]{%
$ 
\begin{array}[1.5\textwidth]{ccc}
\labellist
\pinlabel $\mathcal{P}$ at 129 468
\pinlabel $\mathcal{P}$ at 246 468
\pinlabel $T'$ at 168 452
\pinlabel $T'$ at 168 427 
\pinlabel $F_1$ at 116 439
\pinlabel $F_2$ at 269 439
\pinlabel $S$ at 235 447
\endlabellist
\makebox[2.6in]{\includegraphics[width=2.2in]{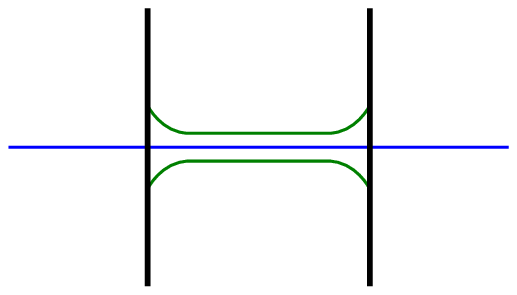} } &

\labellist
\pinlabel $\mathcal{P}$ at 129 468
\pinlabel $\mathcal{P}$ at 246 468
\pinlabel $T'$ at 169 451
\pinlabel $T'$ at 169 417 
\pinlabel $F_1$ at 116 439
\pinlabel $F_2$ at 269 439
\pinlabel $S_+$ at 235 447
\pinlabel $S_-$ at 235 422
\endlabellist
\makebox[2.6in]{\includegraphics[width=2.2in]{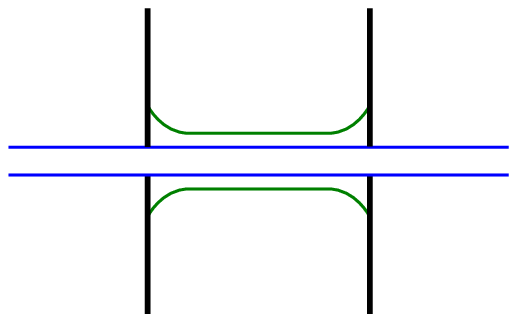} } &

\labellist
\pinlabel $\mathcal{P}$ at 10 70
\pinlabel $\mathcal{P}$ at 10 10
\pinlabel $T'$ at 48 41
\pinlabel $F_1$ at -5 49
\pinlabel $F_2$ at -5 33
\pinlabel $S_+$ at 35 52
\pinlabel {\Large $F \x [0,1]$} at 19 40
\endlabellist
\makebox[1.5in]{\centering \includegraphics[width = 1.1in]{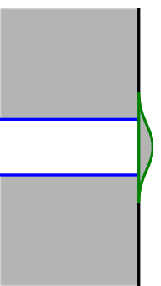} } \\
\end{array} 
$
}

\caption{\label{fig:crimping the paper} Preparing to fold $\mathcal{P} \cup \nu( A)$, we crimp the edges so that we may glue in $F \x [0,1]$.}
\end{figure}

We can then glue in $F_1 \times [0,1]$ as prescribed by the spinal tap. Moreover, since we folded $S$ so as to preserve the dividing set on $A$, after gluing in $F_1 \times [0,1]$ the contact structure is isotopic to an $S^1$-invariant contact structure -- i.e., the ``bump'' we added rounding near $A$ thus just extends the contact structure by a standard product neighborhood. This constructs the contact structure on $\mathcal{P}'$, and since it is $S^1$-invariant, it remains compatible with $\mathcal{P'}$.

\end{proof}

The last piece we need for the proof of Proposition~\ref{prop:unfolding} involves equating convex folding with a sequence of contact $(+1)$-surgeries and $1$-handle removal.  

\begin{lemma} Let $S$ be a spinal tap surface in a contact $3$-manifold $(Y, \mathcal{B}, \xi)$ and let $(Y', \mathcal{B}', \xi')$ be the result of folding along $S$. Then there is a Stein cobordism from $(Y', \xi')$ to $(Y, \xi)$. Moreover, if $\mathcal{B}'$ is the boundary of a Lefschetz fibration $(X',f')$, this fibration can be extended along the Stein cobordism to a Lefschetz fibration $(X,f)$ with boundary $\mathcal{B}$.

\end{lemma}

\begin{proof} (cf. \cite{Baldwin}, \cite{Avdek}) Let $S$ be a convex surface in $(Y, \xi)$ and suppose $S_+$ and $S_-$ are homeomorphic surfaces. Folding $Y$ along $S$ gives a new contact manifold - first we cut $Y$ along $S$ and then we glue in two copies of $S_+\x [0,1]$, one to each boundary component $S_1$ and $S_2$ of $Y \bs S$. As in the definition of the spinal tap, we choose an orientation reversing diffeomorphism $h:S_+ \to S_-$ which preserves the identification of $\partial S_+$ with $\partial S_-$ given by $S$ and glue by the following identifications:
$$\begin{array}{c} S_+ \x \{0\} \xrightarrow{id} S_+ \\ S_+ \x \{1\} \xrightarrow{h} S_- \\
\end{array}$$
on $S \x \{ 1 \} $
and 
$$\begin{array}{c} -S_+ \x \{0\} \xrightarrow{\bar{h}} -S_- \\ -S_+ \x \{1 \} \xrightarrow{id} -S_+ \\
\end{array}$$
on $S \x \{ 0 \} $

Since $S_+ \x [0,1]$ is disk decomposable, we can take a collection of decomposing arcs for $S_+$ and extend them to a collection of decomposing disks for $S_+ \x [0,1]$. Gluing in $S_+ \x [0,1]$, then, is the same as gluing in these decomposing disks and then filling in the remaining $S^2$ boundaries (or boundary, if the collection is minimal) with  standard contact $3$-balls.

We want to compare this with the following surgery construction. Let $\{a_i\}$ be an arc decomposition of $S_+$ and extend each arc into $S_-$ on $S$ by $h(a_i)$. This will give a Legendrian link $\mathcal{L}$ on $S$, each component of which has Thurston-Bennequin number one less than the framing induced by $S$, i.e., $tb = pf -1$. Contact (+1)-surgery is then topological 0-surgery. 

Let $l_i$ be a component of $\mathcal{L}$. A standard neighborhood of $l_i$, $\nu(l_i)$, framed by $S$ is $l_i \x [0,1] \x [0,1]$, where the first $I$ factor is the neighborhood in $S$ and the second gives the vertical direction. The boundary of this neighborhood consists of four annuli, two horizontal and two vertical. Cutting along $S$ is the same as cutting out all of these neighborhoods, plus removing a neighborhood of the resulting complementary punctured disk $S \backslash \mathcal{L}$. Topological 0-surgery along $\mathcal{L}$ glues in a solid torus so that meridional disks get attached to the longitudinal fibers of $\partial \nu(l_i)$. These are the same longitudes as the fibers in the four annuli which make up $\partial \nu(l_i)$. Attaching disks along the horizontal annuli folds the skeleta of $S_1$ and $S_2$ together. Attaching disks along the horizontal annuli caps off the punctures of the complementary disk $S-\mathcal{L}$, yielding an $S^2 \x I$ region. If we further cut along this $S^2$ and glue in two $3$-balls, this gives the result of folding along $S$. 

On the contact side, then, we can think of (+1)-surgery on $\mathcal{L}$ as follows. First we cut along $S$. Then we attach a pair of thickened standard decomposing disks $D^2 \x [0,1]$ along each component $l_i$ of $\mathcal{L}$, one sitting on each boundary $S_1$ and $S_2$. The new boundary is then a pair of standard convex  2-spheres, which get glued together. 

To finish the construction and end with the folded manifold $Y_0$, we need to cut along this convex $S^2$ and fill in with two standard contact $3$-balls --- i.e., we need to remove a standard contact $S^1 \x S^2$.

Thus $Y'$ is built from $Y$ by a sequence of contact $(+1)$-surgeries and a $4$-dimensional $1$-handle removal. The reverse cobordism from $Y'$ to $Y$ consists of a single Weinstein 1-handle, and $b_1(S_+)$ Weinstein 2-handles, which gives a Stein cobordism from $Y'$ to $Y$. 

We can understand the upside down cobordism as well. After folding, $Y'$ has two surfaces $\tilde{S}_1$ and $\tilde{S}_2$ which are naturally convex, have transverse boundary and trivial dividing set. In particular, they are (subsets of) pages of the spinal open book $\mathcal{B}'$. Unfolding consists of removing neighborhoods of these two surfaces and gluing together the resulting convex boundary surface $S$. The 1-handle is easy. The 2-handles are attached along the dual link to $\mathcal{L}$ in $Y' \# S^1 \x S^2$. Each component of the dual link consists of four arcs, $a_1, a_2, a_3, a_4$, each dual to the $D^2 \times I$ subset of the surgery solid torus as described above. The arc $a_1$ lies on $\tilde{S}_1$, $a_3$ lies on $\tilde{S}_2$, and the arcs $a_2$ and $a_4$ run between them across the 1-handle.  

If $\mathcal{B}'$ is already the boundary of a Lefschetz fibration $L'$, then the surfaces $\tilde{S}_1$ and $\tilde{S}_2$ are each fibers of $L' = F\hookrightarrow X' \xrightarrow{\pi} \Sigma'$ and this handle decomposition is precisely the handle decomposition used to extend $L'$ as an $F$-bundle over an additional 1-handle attached to $\Sigma'$. Moreover, the gluing map used to extend $L'$ identifies the arcs $a_1$ on $\tilde{S}_1$ with $a_3$ on $\tilde{S}_2$.

\end{proof}

\noindent The above Lemma concludes the proof of Proposition~\ref{prop:unfolding}.

\vspace{0.1in}
%================================================================================================================
\subsection{Stein structures on Lefschetz fibrations over arbitrary surfaces} \
% ================================================================================================================

With the machinery we have developed in the previous subsections in hand, we can now prove: 

\enlargethispage{0.5in}
\begin{theorem} \label{thm:steinfilling} Suppose $(Y, \mathcal{B})$ is the boundary of an allowable Lefschetz fibration $(X, f)$, and let $\xi$ be the contact structure supported by $\mathcal{B}$. Then $X$ admits a Stein structure $J$ whose convex, contact boundary is $\xi$, i.e. $(X,J)$ is a Stein filling of $(Y,\xi)$.
\end{theorem}

\begin{proof} Let $F \hookrightarrow X \xrightarrow{f} \Sigma$ be a Lefschetz fibration with boundary $\mathcal{B}$. Take a properly embedded arc in $\Sigma$ which is disjoint from the critical values of $f$. Then $\ds S = f^{-1} (a)\big|_{\partial X}$ is a spinal tap surface in $\mathcal{B}$. If we cut $X$ along $f^{-1}(a)$ we get a new Lefschetz fibration $F \hookrightarrow X' \xrightarrow{f'} \Sigma'$ with boundary $\mathcal{B}'$, the spinal open book formed by the spinal tap along $S$. By Proposition~\ref{prop:unfolding}, there is a Stein cobordism which extends $f'$ on $X'$ to $f$ on $X$. If we then take a set of decomposing arcs for $\Sigma$ and cut along their $F \times I$ preimages in $L$, we are left with a Lefschetz fibration over the disk whose boundary is the result of the successive spinal taps on $\mathcal{B}$. Since the resulting Lefschetz fibration over a disk admits a Stein structure filling the boundary open book, repeatedly applying Proposition~\ref{prop:unfolding} proves that $X$ also admits a Stein structure filling its boundary spinal open book $\mathcal{B}$.

\end{proof}

Combining the above theorem with the stronger result of Loi and Piergallini (See Theorem~\ref{PALF}), we get:

\begin{corollary} \label{LPgeneralization}
An oriented compact $4$--manifold with boundary is a Stein surface, up to orientation preserving diffeomorphisms, if and only if it admits an allowable Lefschetz fibration over a compact surface with non-empty boundary. Moreover, any two allowable Lefschetz fibrations filling the same spinal open book carry Stein structures which fill the same contact structure induced by the spinal open book. 
\end{corollary}

In our constructions to follow we make repeated use of the following corollary, which generalizes a result of Akbulut and Ozbagci \cite{AO2}:

\begin{corollary} \label{LFextract}
Let $X$ be a 4-manifold, closed or with boundary, and $f\colon X \to \Sigma$ be an allowable Lefschetz fibration over any compact surface $\Sigma$, closed or bounded, $F$ a regular fiber and $S_1, \ldots, S_m \subset \text{Int}(X) \setminus \text{Crit}(f)$ a non-empty collection of \emph{disjoint} sections of this fibration. Let $X_0$ be the $4$-manifold we obtain from $X$ by excising fibered tubular neighborhoods of $F, S_1, \ldots, S_m$. Then $X_0$ admits a Stein structure. In particular, this holds when $f\colon X \to \Sigma$ is a Lefschetz fibration on a closed $4$-manifold $X$ and none of the Lefschetz vanishing cycles are separating. Moreover, if $f\colon X \to \Sigma$ and $f'\colon X' \to \Sigma$ are any two allowable Lefschetz fibration over a closed surface $\Sigma$ with regular fibers $F \cong F'$, and with disjoint sections $S_1, \ldots, S_m$ and $S'_1, \ldots, S'_m$ of matching self-intersection numbers, then $X_0= X \setminus (F \cup S_1 \cup \ldots S_n)$ and $X'_0= X' \setminus (F \cup S'_1 \cup \ldots S'_n)$ admit Stein structures inducing the same contact structure on their identified boundaries.
\end{corollary} 

\begin{proof} 
When we remove a fiber and a collection of disjoint sections from an allowable Lefschetz fibration (and in particular from a Lefschetz fibration with no separating vanishing cycles), we are left with another allowable Lefschetz fibration. The boundaries of the induced Lefschetz fibrations on $X_0$ and $X_0'$ are isomorphic spinal open books with spine $\whs = (\Sigma - D^2)_1 \cup \dots \cup (\Sigma - D^2)_m$ and page $F - ( D^2_1 \cup \dots \cup D^2_m )$. So the statements follow from the previous theorems. 
\end{proof}

\newpage 
%	================================================================================================================
\section{Contact $3$-manifolds admitting arbitrarily large Stein fillings} \label{ConstructLFs}
% ================================================================================================================
% ================================================================================================================

% ================================================================================================================
\subsection{Main construction} \
% ================================================================================================================

We are going to produce the families of Stein fillings promised in Theorem~\ref{mainthm}, by first engineering certain families of Lefschetz fibrations with distinguished sections. 

\begin{theorem} \label{LFs}
Let $g\geq 2$, $h \geq 1$, and $n \leq 2h-2$ be fixed integers. For any positive $m$, there is a genus $g$ Lefschetz fibration $(X(m), f(m))=(X_{g,h,n}(m), f_{g,h,n}(m))$ over a genus $h$ surface, such that 
\begin{enumerate}
\item $(X(m), f(m))$ has only non-separating Lefschetz vanishing cycles, and the rank of the critical locus $M(m)=M_{g,h,n}(m)$ is strictly increasing in $m$. \ 
\item $(X(m), f(m))$ admits a section $S_n=S_{g,h,n}(m)$ of self-intersection $n$. \
\end{enumerate}
Moreover, when $g=2$, for any fixed $h \geq 1$, the signature of $X(m)=X_{2,h,n}(m)$ is strictly decreasing in $m$.
\end{theorem}

\noindent For any section of a genus $g$ Lefschetz fibration over a genus $h \geq 1$ surface, its self-intersection number is determined by the number of critical points when $g=1$, and is bounded above by $2h-2$, when $g \geq 2$ and $h \geq 1$, as shown in \cite{BKM}. So the triples $(g,h,n)$ realized in the theorem above are all one can possibly get.

\begin{proof} 
We will construct the families of Lefschetz fibrations and sections prescribed in the statement using factorizations in the mapping class groups of surfaces. As outlined in Section~\ref{Prelim}, we need to obtain relations in $\Gamma_g^1$ of the form
\begin{eqnarray*} 
t_{\delta}^{-n} &=& \text{Product of $M(m)$ positive Dehn twists along non-separating curves} \\
& & \text{and of $h$ commutators} 
\end{eqnarray*}
where $n$ is the self-intersection of a section $S_n$ and $M(m)$ is a multivariable function depending on $g,h,n,m$, which is strictly increasing in $m > 0$. 

Let $g\geq 2$ and $h\geq 1$ be fixed integers. All the relations below should be understood to take place in $\Gamma_g^1$. Our key input is the following family of relations obtained in \cite{BKM} (See proof of Theorem~21; relations 12--20). See Figure~\ref{relationcurves} for the curves that appear below.

When $h=1$, the following relation holds for any positive integer $m$
\begin{eqnarray} \label{h=1}
t_\delta^{0} =1 = \, C(m) \, T^m \, ,
\end{eqnarray}
where $C(m)$ is a single commutator that depends on $m$, and 
\[ T= t_{c_2} t_{c_1} (t_{c_1} t_{c_2} t_{c_3})^2 t_{c_1} t_{c_2} \]
is a product of positive Dehn twists. (See relation~20 of \cite{BKM}; here we chose $l=2$.) Note that $c_2=b$ for $g=2$.

Whereas for $h>1$, for any positive integer $m$ we have
\begin{eqnarray} \label{h>1}
 t_\delta^{2-2h} &=& C_1 \cdots C_{h-1} \, C{(m)} T_1^m \, T_2^m \,    ,
\end{eqnarray}
where $C_1, \ldots, C_{h-1}$ are fixed commutators, $C(m)$ is a single commutator that depends on $m$. Here
\[ T_1 = t_{r} t_{a_1} t_{b} t_{r} ( t_{a_1}t_{r}t_{b})^2 \, \]
is a product of positive Dehn twists, and from 
\[ T_2 ( t_{c_1} t_{c_2} \cdots t_{c_{2g-3}})^{2g-2}= (t_{c_1} t_{c_2} \cdots t_{c_{2g-3}} t_{c_{2g-2}} t_{b})^{2g} \]
one obtains $T_2$ as a product of $8g-6$ positive Dehn twists. (See the paragraph following Equation~17 in \cite{BKM}.)

\begin{figure}[hbt]
 \begin{center}
     \includegraphics[width=14cm]{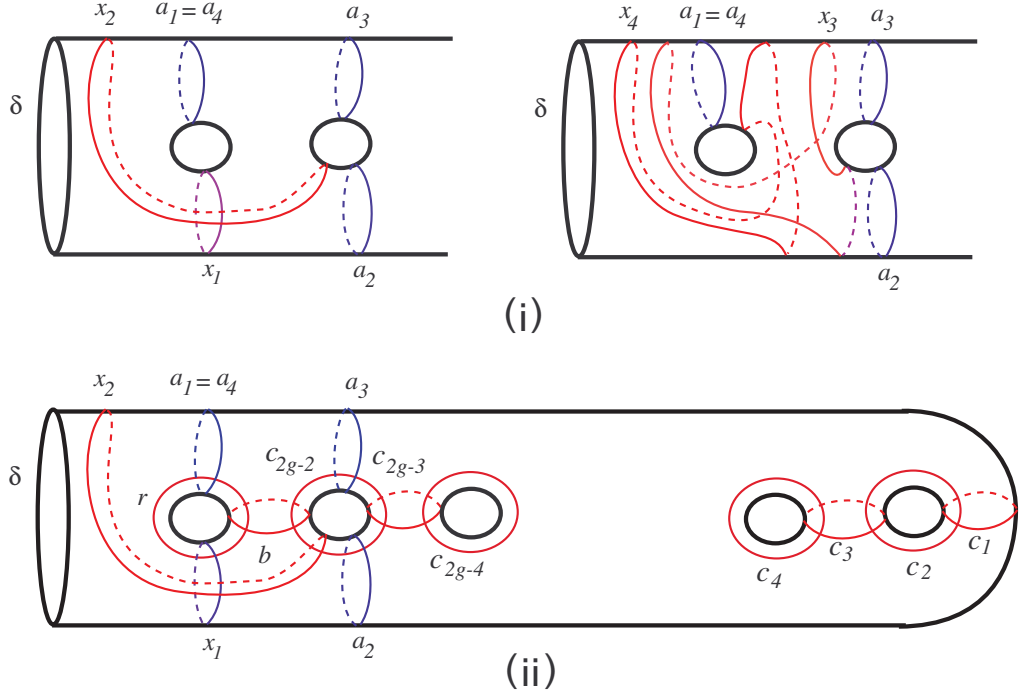}
     \caption{The curves of the mapping class group relations. When $g=2$, we have $b=c_{2g-1}$ and $r=c_{2g}$.}
     \label{relationcurves}
 \end{center}
\end{figure}

On the other hand, we have the (one boundary) chain relation 
\begin{eqnarray*} \label{chain1}
 t_{\delta} &=& (t_{c_1} t_{c_2} \cdots t_{c_{2g-3}} t_{c_{2g-2}} t_{b} t_r)^{4g+2} \, .
\end{eqnarray*}
Let $R$ denote the product of positive Dehn twists appearing on the right hand side of this relation, so it contains $8g^2+4g$ Dehn twists. We can multiply the two sides of the equations (\ref{h=1}) and (\ref{h>1}) by $t_{\delta}^k$ and $R^k$ to get

\begin{eqnarray} \label{h=1final}
t_\delta^{k} = \, C(m) \, T^m \, R^k \, , \text{when $h=1$, and}
\end{eqnarray}
\begin{eqnarray} \label{h>1final}
 t_\delta^{2-2h+k} &=& \, C_1 \cdots C_{h-1} \, C{(m)} \, T_1^m \, T_2^m \, R^k  \, , \text{when $h>1$.}
\end{eqnarray}
So both relations prescribe genus $g$ Lefschetz fibrations over genus $h$ surfaces with sections of self-intersection $n=2h-2-k$ and with only non-separating vanishing cycles.  

The number of Lefschetz critical points $M(m)=M_{g,h,n}(m)$ can be calculated as
\[ M(m) =\left\{ \begin{array}{ll}
10m+(8g^2+4g)k = 10m-(8g^2+4g)n & \mbox{for $h=1$} \\
(8g+4)m+(8g^2+4g)k= (8g+4)m+(8g^2+4g)(2h-2-n) & \mbox{for $h>1$} 
\end{array}
\right. \]
which is strictly increasing in $m$ for any $g, h, n$. Thus, the same holds for 
\[ 
\eu(X(m))= 4(g-1)(h-1) + M(m) \, .
\]

The signatures of the $4$-manifolds $X(m)=X_{g,h,n}(m)$ can be calculated from the explicit monodromy factorizations (\ref{h=1final}) and (\ref{h>1final}) above, by looking at their images under the boundary capping homomorphism $\Gamma_g^1 \to \Gamma_g$. We will carry out this calculation in a simpler case, when the fibration is hyperelliptic, in which case the following signature formula of Endo's \cite{Endo} comes handy:
\begin{equation*}
\sigma(X)=-\frac{g+1}{2g+1}N+
\sum_{j=1}^{[\frac{g}{2}]}(\frac{4j(g-j)}{2g+1} -1)s_j .
\end{equation*}
Here $X$ is the total space of the hyperelliptic fibration, $N$ and $s=\sum_{j=1}^{[\frac{g}{2}]}s_j$ are the numbers of nonseparating and separating vanishing cycles, respectively, whereas $s_j$ denotes the number of separating vanishing cycles which separate the surface into two subsurfaces of genera $j$ and $g-j$. 

When $g=2$, the mapping class group $\Gamma_2$ is hyperelliptic, and thus the genus two fibrations are guaranteed to be hyperelliptic. (Indeed, the reader can check that, in this case, all the curves on the closed surface isotope to curves which are symmetric under the hyperelliptic involution, whereas for $g>2$ they do not.) So we calculate the signature of $X(m)=X_{2,h,n}(m)$ as

\[ \sigma(X(m)) =\left\{ \begin{array}{ll}
-\frac{3}{5} (10m- 40 n) + 0 = -6m + 24 n & \mbox{for $h=1$} \\
-\frac{3}{5} (20m +40(2h-2-n)) + 0 = -12m - 24(2h-2-n) & \mbox{for $h>1$} 
\end{array}
\right. \]
which for any $h \geq 1$ is seen to be strictly decreasing in $m$.

\end{proof}

Now let $Y_{g,h,n}$ be the graph $3$-manifold described in Figure~\ref{graph}. The next theorem provides the promised families of contact $3$-manifolds and their Stein fillings.

\begin{figure}[hbt] 
 \begin{center}
      \includegraphics[width=10cm]{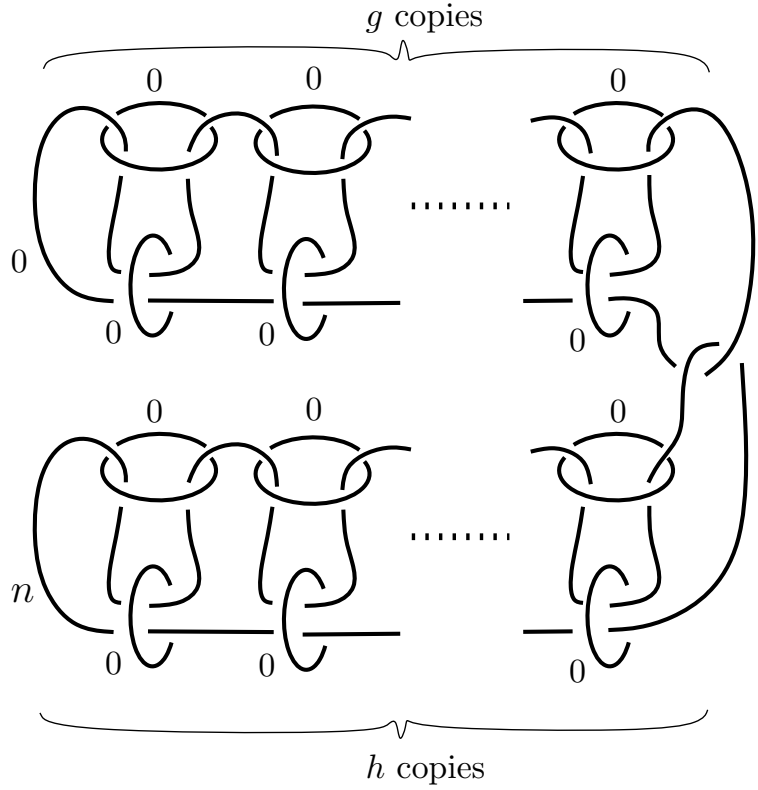}
      \caption{A surgery description of the graph manifold $Y_{g,h,n}$ as a plumbing of a circle bundle over $\Sigma_g$ with euler number $0$ and a circle bundle over $\Sigma_h$ with euler number $n$. The linking patterns are repeated $g$ times on the top and $h$ times on the bottom.}
  \end{center} \label{graph}
\end{figure} 

\begin{theorem}\label{Steinfillings}
Let $g\geq 2$, $h \geq 1$, and $n \leq 2h-2$ be fixed integers. Then $Y_{g,h,n}$ admits a contact structure $\xi_{g,h,n}$, which admits an infinite sequence of Stein fillings $(X(m), J(m))=(X_{g,h,n}(m), J_{g,h,n}(m))$, for $m=0,1, \ldots$, such that the euler characteristic of $X(m)$ is increasing in $m$. Moreover, when $g=2$, for any fixed $h \geq 1$ and $n$, the signature of $X(m)$ is decreasing in $m$.
\end{theorem}

\begin{proof}
From the above theorem, we have a family of Lefschetz fibrations 
\[(X(m), f(m))=(X_{g,h,n}(m), f_{g,h,n}(m))\]
with distinguished sections $S=S_{g,h,n}(m)$ of self-intersection $n$. Removing fibered neighborhoods of a regular fiber and the section $S$ of $(X(m), f(m))$ hands us an allowable Lefschetz fibration $(\check{X}(m), \check{f}(m))$ which induces a framed spinal book $\mathcal{B}_{g,h,n}$ on its boundary $Y_{g,h,n}$, which is fixed for any $m=0,1, \ldots$. By Proposition~\ref{prop:welldefined}, $Y_{g,h,n}$ admits a unique contact structure $\xi_{g,h,n}$ compatible with the spinal open book $\mathcal{B}_{g,h,n}$. On the other hand, by Corollary~\ref{LFextract}, $\check{X}(m)$ admits a Stein structure $J(m)$ filling the contact structure $\xi_{g,h,n}$ on $Y=Y_{g,h,n}$.

The euler characteristics and signatures of $X(m)$ and $\check{X}(m)$ are related by the formulae
\[ \eu(X(m))= \eu(\check{X}(m)) + 3-2(g+h) , \text{and} \]
\[ \sigma(X(m)) = \sigma(\check{X}(m)) + 0 \, , \]
where the latter follows from the Novikov additivity. Therefore we see that $\eu(\check{X}(m))$ is strictly increasing in $m$, and for $g=2$, the $\sigma(\check{X}(m))$ is strictly decreasing. This completes the proof.
\end{proof}

The proof of Theorem~\ref{mainthm} immediately follows:

\begin{proof}[Proof of Theorem~\ref{mainthm}]
For $g=2$, $h \geq 1$, and $n \leq 2h-2$, $(Y_{2,h,n}, \xi_{2,h,n})$ admits Stein fillings $(\check{X}(m),J(m))=(\check{X}_{2,h,n}(m), J_{2,h,n}(m))$ such that $\{ \eu(\check{X}(m)) \}$ is a strictly increasing sequence, and $\{\sigma(\check{X}(m))\}$ is a strictly decreasing sequence, for $m=0,1, \ldots$. So for any given pair of integers $E,S$, there exits a positive integer $P$ such that the infinite subsequence $\{(\check{X}(m), J(m))\}_{m \geq P}$ consists of members whose euler characteristics are greater than $E$ and signatures are smaller than $S$.
\end{proof}

\begin{remark}
We shall note that when discussing the signatures, we restricted ourselves to families with $g=2$ above for brevity. Otherwise, it is possible to see that the signature of $X_{g,h,n}(m)$ is decreasing in $m$ for any fixed $g > 2$, $h \geq 1$, and $n \leq 2h-2$ as well, which however requires a significantly more tedious calculation, since the fibrations we obtain in this case are not hyperelliptic.
\end{remark}

We also note that

\begin{corollary} \label{infinitec1}
There are infinite families of contact $3$-manifolds, where each contact $3$-manifold admits Stein filling with infinitely many different chern numbers $c_1^2$ and $c_2$. 
\end{corollary}

\begin{proof}
We calculate $c_1^2(X)= 2 \eu(X) + 3 \sigma(X)$ of the Stein fillings of $Y_{2,h,n}$ given in the proof of Theorem~\ref{Steinfillings} as $2m-8n$ for $h=1$ and $4m + 8 (2h-2-n)$ for $ h \geq 2$, which constitute an infinite family for varying $m \geq 0$. Since $c_2(X)=\eu(X)$, the latter claim is already proved above.
\end{proof}

\begin{remark}
Filling the fiber component in our allowable Lefschetz fibrations above with $n< 2-2h$, we get new $4$-manifolds whose boundaries are non-flat circle bundles over a closed surface $\Sigma_h$ of genus $h \geq 1$. However, in \cite{S}, Stipsicz showed that \textit{any} contact structure on a non-flat circle bundle over a surface $\Sigma_h$ admits at most finitely many Stein fillings, which implies that the cobordisms we get this way can never be Stein. 
\end{remark}

% ================================================================================================================

\subsection{Further constructions} \
% ================================================================================================================

We will now outline how to obtain similar families of contact structures on more general $3$-manifolds, admitting Stein fillings which have arbitrarily big euler characteristics and arbitrarily small signatures, 

\noindent \textbf{More general graph manifolds.}  We can generalize the above construction to many more graph manifolds, by removing more than one fiber and/or using Lefschetz fibrations with many disjoint sections, and following the same steps as above. The former is straightforward: We can simply remove fibered tubular neighborhoods of $k$ disjoint fibers for $k \geq 2$ to obtain more general graph manifolds that can be described by a surgery diagram similar to the one given in Figure~\ref{graph}, where we will instead have $k$ copies of the top part of the diagram, each one of which linking the bottom part once. Therefore, the same families of Lefschetz fibrations $X_{g,h,n}(m)$ in Theorem~\ref{LFs} can be employed to obtain arbitrarily big Stein fillings of the contact structures given by the spinal open books on these more general graph manifolds.

We can also consider graph manifolds which can be described by a surgery diagram similar to the one given in Figure~\ref{graph}, where this time we would have $l$ copies of the \emph{bottom} part of the diagram, each one of which linking the top part once. However, we are now in need a sequence of Lefschetz fibrations with increasing euler characteristics (and decreasing signatures) \textit{which have $l$ disjoint sections}. Such families can be deduced from the ones we presented in Theorem~\ref{LFs} as follows: Consider the family $(X_{g,1,0}(m), f_{g,1,0}(m))$, for any fixed $g\geq 2$. Each one of these Lefschetz fibrations has a section $S$ of self-intersection $0$. By taking $l$ disjoint push-offs of $S$, we get $l$ disjoint sections of this Lefschetz fibration. We can then take the fiber sum of $(X(m),f(m))=(X_{g,1,0}(m), f_{g,1,0}(m))$ with any genus $g$ Lefschetz fibration over the $2$-sphere with $l$ disjoint sections $S_1, \ldots, S_l$ of self-intersections $r_1, \ldots, r_l$, with only non-separating vanishing cycles. Possibly after an isotopy, we can patch the disjoint sections coming from both summands so as to get a new family of Lefschetz fibrations $(X'(m), f'(m))$ with $l$ disjoint sections $S'_1, \ldots, S'_l$ of self-intersections $r_1, \ldots, r_l$. As before, we see that the euler characteristic of $X'(m)$ is strictly increasing in $m$ (and its signature for $g=2$ is strictly decreasing). Hence, excising fibered neighborhoods of a regular fiber and these $l$ sections, we obtain the desired Stein fillings of the $3$-manifold on the boundary, equipped with the natural contact structure induced by the spinal book. It is worth noting that there are many examples of Lefschetz fibrations with disjoint sections of \emph{different} self-intersections. Thus we can obtain graph manifolds where the framings $r_1, \ldots, r_l$ on the $l$ copies mentioned above are not necessarily the same. Lastly, we can push for even more general families of graph manifolds by taking out more than one fiber in these Lefschetz fibrations as before.

\noindent \textbf{Non-graph manifolds.} It is also possible to generalize our constructions to the case of Stein fillable contact $3$-manifolds supported by spinal open books whose page monodromies are non-trivial --- which typically will hand us non-graph manifolds. 

Let $f: X \to \Sigma$ be a (not necessarily allowable) Lefschetz fibration with regular fiber $F$ and base $\Sigma$ compact surfaces with non-empty boundary and $Y = \partial X$. Similar to the description of a standard open book, $f|_{Y}$ gives a spinal open book $\mathcal{B}$. The paper of this spinal open book is vertical boundary of $X$, $\mathcal{P} = f^{-1} (\partial D)$. The spine is the complementary region $Y \backslash \mathcal{P}$, and is the horizontal boundary of $X$. The Lefschetz fibration then equips $\mathcal{S}$ with the structure of a circle bundle with fibers consisting of the boundaries of all fibers of $f$, $\partial F$, as a bundle over the base $\Sigma$. If $\partial F$ is disconnected, then the vertebra $\whs$ consists of $\#|\partial F|$ copies of $\Sigma$. As the boundary of a Lefschetz fibration, $\mathcal{B}$ is a symmetric, uniform, simple spinal open book: every fiber is isotopic within $X$ and so every component of the fiber $\whf$ of $\mathcal{B}$ is isomorphic. The spine consists of the bundle of the disconnected union of circle boundaries of the fibers, and as such circles are isotopic within each horizontal boundary to the boundary of a single fiber, there is a single boundary component of each component of the total fiber $\whf$ which gets glued to a component of $\whs$.

For a standard open book $\mathcal{B}$, there exists a Lefschetz fibration with boundary $\mathcal{B}$ if we can find a factorization of the monodromy of $\mathcal{B}$ into positive Dehn twists. Any such factorization gives a Lefschetz fibration filling of $\mathcal{B}$. For spinal open books, the picture is slightly more complicated:

Given a spinal open book $\ds \mathcal{B} = (Y, \whf, \widehat{\phi}, \whs, g)$, there exists a Lefschetz fibration with boundary $\mathcal{B}$ if we can find identifications $\hat{i}: \whf \to F$ and a factorization of the \emph{total monodromy} $\Phi_{\hat{i}} = \phi_1^{i_1} \circ \cdots \circ \phi_n^{i_n}$ (where $\phi^i = i \circ \phi \circ i^{-1}$) as 
$$\Phi_{\hat{i}} =  \prod_{i = 1}^m t_i \prod_{j = 1}^h [\alpha_j, \beta_j]$$
in the mapping class group of $F$, 
where $h$ is the genus of $\Sigma$, $a_j$,$b_j$,  are isotopy classes of diffeomorphisms of $F$ for $j = 1, \dots, h$, and $t_i$, $i =1, \dots, m$ are Dehn twists on $F$. In particular, such a factorization corresponds to the monodromy presentation of the bundle of non-singular fibers in a Lefschetz fibration with this boundary. 

Given any mapping class element $\Phi$, we define the \emph{positive coset commutator length} of $\Phi$ to be the smallest $h$ so that we can write $\Phi$ as a the product of a length $h$ commutator and positive Dehn twists as above. 

%To do so, observe that in our proof of Theorem~\ref{LFs}, we increased the number of positive Dehn twists by taking the powers of the elements $T, T_1$, or $T_2$, but not $R$. Therefore, one can take a $2$-disk $D$ containing any subcollection of the singular fibers introduced by $R$ in the factorization, and consider the Lefschetz fibration over $\Sigma_h \setminus D$. We can then remove a fibered neighborhood of the (restriction of the) section $S$ so as to get a new allowable Lefschetz fibration. Clearly, this Lefschetz fibration has non-trivial monodromy, which in turn implies that the spinal open book on its boundary has non-trivial monodromy. The reader can easily verify that the rest of our arguments carry out in this more general setting as well.

\begin{theorem} Let $\mathcal{B}$ be a symmetric, uniform, simple spinal open book with page $F$ of genus greater than two, spine $\Sigma$ of genus $h$. If there is a set of identifications $\hat{i}$ so that the positive coset commutator length of the total monodromy $\Phi_{\hat{i}}$ is strictly less than $h$, then $\xi_\mathcal{B}$ admits Stein fillings of arbitrarily large euler characteristic.
\end{theorem}

\begin{proof} If there is such a total monodromy $\Phi_{\hat{i}}$ with commutator length strictly less then $h$, then in the monodromy presentation of the associated Lefschetz fibration, we can choose a single commutator to be that of the identity maps. We can extend this factorization to new Lefschetz fibrations by making a monodromy substitution using the relations given in (2) above so as to produce arbitrarily big allowable Lefschetz fibrations filling $\mathcal{B}$ as before. Thus, any contact $3$-manifold satisfying these properties will admit arbitrarily big Stein fillings.
\end{proof}

% ================================================================================================================
\subsection{Final remarks} \
% ================================================================================================================

We finish with a couple of intriguing questions that arise in our work:

\noindent \textbf{Positive factorizations.} It follows from the Theorem~\ref{PALF} that the Stein fillings $(X(m), J(m))=(X_{g,h,n}(m), J_{g,h,n}(m))$ of $(Y, \xi)=(Y_{g,h,h}, \xi_{g,h,n})$ we constructed in the previous section admit allowable Lefschetz fibrations \textit{over the $2$-disk} whose boundary open books support $(Y, \xi)$. In turn, these prescribe positive open books supporting the same contact structure $(Y, \xi)$. Given that the euler characteristic of $X(m)$ grows along with the number of Lefschetz critical points as $m$ grows, understanding how the induced positive factorizations are correlated is an intriguing task which we address in a future work. There are various ways to see that these positive open books will be non-planar. In fact, the following observation by John Etnyre and Amey Kaloti presents a nice contrast: For any closed contact $3$-manifold $(Y, \xi)$ supported by a planar open book, there is a bound on the possible euler characteristics and signatures realized by minimal symplectic fillings of it. (See \cite{Ka}.) In other words, the conjecture of Stipsicz and Ozbagci holds in this case.

\noindent \textbf{Underlying geometries.} A curious point that arises in our work is as follows: There are Stein fillable contact $3$-manifolds which admit (1) a unique Stein filling, (2) more than one but finitely many Stein fillings, and (3) infinitely many Stein fillings, up to diffeomorphisms. That is, there is an intrinsic property one can associate to Stein fillable contact $3$-manifolds in terms of the number of Stein fillings they admit. There are examples of contact $3$-manifolds which carry only one of these properties. For example, Gromov \cite{Gromov} proved that there is a unique minimal symplectic filling of $S^3$, whereas McDuff \cite{McDuff} proved that there are exactly two minimal fillings of the standard contact structure on $L(4,1)$. Ozbagci and Stipsicz in \cite{OS} showed that the Seifert manifolds with a single Seifert fiber of order 2 and with base a surface of genus $g\geq2$ admit a contact structure with infinitely many Stein fillings; also see \cite{AEMS}. In this article, we have shown that there is a fourth class of Stein fillable contact $3$-manifolds, namely those which admit (4) infinitely many Stein fillings with arbitrarily large euler characteristics. It is therefore worth asking whether or not there are Stein fillable contact $3$-manifold which belong to the class (3) but not (4). To the authors of this article, this question seems to have a strong tie with the underlying geometry of the $3$-manifold: All the earlier examples of Stein fillable contact $3$-manifolds that belong to class (3) are Seifert fibered ones, whereas the class (4) examples we produce are non-geometric.

\vspace{0.2in}
\subsection*{Acknowledgments} The first author was partially supported by the NSF grant DMS-0906912. The authors would like to thank the organizers of the 2012 Georgia Topology Conference for the wonderfully stimulating atmosphere they created, where this work was shaped. We thank John Etnyre, Burak Ozbagci, and Chris Wendl for their comments on the first version of this paper. We would also like to thank John Etnyre and Amey Kaloti for pointing out their observation on planar open books, and Sam Lisi and Chris Wendl for the appendix they have written for our article.

\vspace{0.3in}
%\include{appendix}

% ====================================================================================================================

\refstepcounter{secnumdepth}
\refstepcounter{section}

\section*{Appendix (by Samuel Lisi and Chris Wendl): Stein structures on Lefschetz
fibrations and their contact boundaries}

\renewcommand{\thesection}{A}

In this appendix we explain a special case of a theorem 
from \cite{LVW} which implies that an allowable
Lefschetz fibration over an arbitrary oriented surface with boundary can always
be viewed in a canonical way as a Stein filling of a contact structure determined 
by the spinal open book at the boundary (cf.~Theorem~\ref{thm2}).
Our proof is a variation on the technique of Thurston
\cite{Thurston} and Gompf \cite{GS} for constructing
symplectic structures on Lefschetz fibrations.

Let $E$ be a smooth, compact, oriented and
connected $4$-manifold with boundary and corners such that $\p E$ is the
union of two smooth faces
$$
\p E = \p_h E \cup \p_v E
$$
intersecting at a $2$-dimensional corner.  Let $\Sigma$ denote a compact,
oriented and connected surface with nonempty boundary.  We consider
a Lefschetz fibration $\Pi \colon E \to \Sigma$ with the following properties:
\begin{enumerate}
\item The sets of critical points $E\crit$ and critical values
$\Sigma\crit$ lie in the interiors of $E$ and $\Sigma$ respectively,
\item $\Pi^{-1}(\p\Sigma) = \p_v E$ and $\Pi|_{\p_v E} \colon \p_v E \to \p\Sigma$
is a smooth fiber bundle,
\item $\Pi|_{\p_h E} \colon \p_h E \to \Sigma$ is also a smooth fiber bundle,
\item All fibers $E_z := \Pi^{-1}(z)$ for $z \in \Sigma$ are connected and have
nonempty boundary in $\p_h E$.
\end{enumerate}
As we will review in \S\ref{sec:appendixStein} below, 
any Lefschetz fibration of this type induces a spinal open book at its boundary.
We say that $\Pi$ is \defin{allowable} if all the irreducible components of 
its fibers have nonempty boundary, i.e. none of its vanishing cycles are
homologically trivial.

\begin{theorem}
\label{thm:LefschetzStein}
If the Lefschetz fibration $\Pi \colon E \to \Sigma$ is allowable, then after
smoothing corners on $\p E$, $E$ admits (canonically up to 
Stein homotopy) the structure
of a Stein domain, and the filled contact structure at the boundary
is uniquely determined up to isotopy by the induced spinal open book.
\end{theorem}

In the background of
this theorem is the corresponding existence and uniqueness result 
(also a special case of a theorem in \cite{LVW}) 
for contact structures supported by spinal open books.  We shall state and
prove this in \S\ref{sec:appendixSpinal}, and then prove
Theorem~\ref{thm:LefschetzStein} in \S\ref{sec:appendixStein}.

\begin{remark}
A version of Theorem~\ref{thm:LefschetzStein} also holds without the
allowability assumption, but in that case $E$ generally becomes a strong
symplectic filling instead of a Stein filling.  See \cite{LVW}
for details.
\end{remark}

\subsection*{Acknowledgments}
We thank Kai Cieliebak for helpful conversations.
SL was partially supported by the ERC Starting Grant of Fr\'ed\'eric Bourgeois
StG-239781-ContactMath.  CW was partially supported by a Royal Society
University Research Fellowship.

\subsection{Spinal open books and contact structures}
\label{sec:appendixSpinal}

To establish notation,
we begin by reviewing some essential definitions (cf.~Section~\ref{Spinal}).

\begin{defn}
\label{defn:spinal}
A \defin{spinal open book decomposition} on a closed oriented
$3$-manifold~$M$ is a decomposition $M = M\spine \cup M\paper$, where the 
pieces $M\spine$ and $M\paper$ (called the \defin{spine} and \defin{paper} 
respectively)  are smooth compact $3$-dimensional submanifolds
with disjoint interiors such that $\p M\spine = \p M\paper$,
carrying the following additional structure:
\begin{enumerate}%[(i)]
\item A smooth fiber bundle $\pi\spine \colon M\spine \to \Sigma$ with fiber $S^1$,
such that each fiber is either disjoint from $\p M\spine$ or contained in it.
Here, $\Sigma$ is a compact oriented surface whose connected components 
(called \defin{vertebrae}) all have nonempty boundary.
\item A smooth fiber bundle $\pi\paper \colon M\paper \to S^1$ such that the
connected components (called \defin{pages}) of fibers are all compact
surfaces with nonempty boundary, where they meet $\p M\paper$ transversely.
Moreover, the boundary components of each page are fibers of~$\pi\spine$.
\end{enumerate}
\end{defn}

We shall denote by
$$
\boldsymbol{\pi} := \Big(\pi\spine \colon M\spine \to \Sigma,
\pi\paper \colon M\paper \to S^1\Big)
$$
the collection of information encoded in a spinal open book.
We will say additionally that $\boldsymbol{\pi}$ admits
a \defin{smooth overlap} if the fibration $\pi\paper \colon M\paper \to S^1$
can be extended over an open neighborhood $M\paper' \subset M$ containing
$M\paper$ such that all fibers of $\pi\spine$ intersecting $M\paper'$ are
contained in fibers of the extended~$\pi\paper$.
Note that while an arbitrary spinal open book does not always admit a smooth
overlap, it can always be deformed continuously to one that does, and the
result is unique up to smooth isotopy.  

\begin{defn}
\label{defn:supported}
Given a spinal open book $\boldsymbol{\pi}$ on~$M$,
a positive contact form $\alpha$ on~$M$ will be called a
\defin{Giroux form} for $\boldsymbol{\pi}$ if the following conditions
hold:
\begin{enumerate}
\item The 2-form $d\alpha$ is positive on the interior of every page.
\item The Reeb vector field $R_\alpha$ determined by $\alpha$
is positively tangent to every
oriented fiber of $\pi\spine \colon M\spine \to \Sigma$.
\end{enumerate}
A contact structure $\xi$ on~$M$ is \defin{supported} by 
$\boldsymbol{\pi}$ whenever it admits a contact form which
is a Giroux form.
\end{defn}

\begin{theorem}
\label{thm:GirouxForms}
If $\boldsymbol{\pi}$ is a spinal open book on~$M$ which admits a smooth
overlap, then the
space of Giroux forms for $\boldsymbol{\pi}$ is nonempty and contractible.
In particular, any isotopy class of spinal open books gives rise to a
canonical isotopy class of supported contact structures.
\end{theorem}

\begin{remark}
One can also formulate the above definitions and prove a generalization of
Theorem~\ref{thm:GirouxForms} for compact manifolds with boundary,
which allows for a useful alternative characterization of certain ``local''
filling obstructions such as Giroux torsion and planar torsion,
see~\cite{LVW} for details.
\end{remark}

The proof of Theorem~\ref{thm:GirouxForms}
will occupy the remainder of this subsection.
As a first step, we define a \defin{fiberwise Giroux form} for 
$\boldsymbol{\pi}$ to be any smooth $1$-form $\alpha$ on~$M$ for which the 
following conditions hold:
\begin{itemize}
\item $d\alpha$ is positive on the interior of every page,
\item $\alpha$ is positive on the fibers of $\pi\spine \colon M\spine \to \Sigma$,
and the tangent spaces to these fibers are contained in $\ker d\alpha$.
\end{itemize}
A fiberwise Giroux form is a Giroux form if and only if it is contact,
but since we have not required the latter in the above definition,
the space of fiberwise Giroux forms is convex.

Choose for each connected component 
of $\p\Sigma$ a collar neighborhood $(-1,0] \times S^1$ with
coordinates $(s,\phi)$, and enlarge $\Sigma$ by attaching 
$[0,1) \times S^1$ in the obvious way to each of these collars, 
denoting the resulting
surface by $\widehat{\Sigma}$.  If $\boldsymbol{\pi}$ admits a smooth
overlap, then this can be done so that there is also an
open neighborhood $\uU\spine$ of $\p M\spine$ in $M$ which we can
identify with $(-1,1) \times \p M\spine$ such that the fibration
$$
\uU\spine = (-1,1) \times \p M\spine \to (-1,1) \times S^1 \colon
(s,x) \mapsto (s,\pi\paper(x))
$$
matches $\pi\spine$ on $\uU\spine \cap M\spine$, which is the
region $\{ s \le 0 \}$.  In fact, this defines
an extended fibration
$$
\hat{\pi}\spine \colon \widehat{M}\spine \to \widehat{\Sigma},
$$
where $\widehat{M}\spine := M\spine \cup \uU\spine$.
We shall continue to denote the coordinates on the collars
$(-1,1) \times S^1 \subset \widehat{\Sigma}$ by $(s,\phi)$ and, in light
of the compatibility of the two fibrations, also use
$\phi \in S^1$ to denote the coordinate on the base of
$\pi\paper \colon M\paper \to S^1$.

Choose a Liouville form
$\sigma$ on $\widehat{\Sigma}$ that matches $e^s\, d\phi$ on the collars
$(-1,1) \times S^1$.  For convenience, we can also fix an identification
of the (necessarily trivial) bundle $\widehat{M}\spine \to \widehat{\Sigma}$ 
with $\widehat{\Sigma} \times S^1$ such that $\hat{\pi}\spine(z,\theta) = z$.
This identifies each connected component of $\uU\spine$ with
$(-1,1) \times S^1 \times S^1$, carrying coordinates $(s,\phi,\theta)$.
In these coordinates on the collar $M\paper \cap \uU\spine \cong
[0,1) \times S^1 \times S^1$ we have $\pi\paper(s,\phi,\theta) = \phi$.

To keep orientations straight, it will also be convenient to
define an alternative coordinate system on $M\paper \cap \uU\spine$ by
$$
(t,\phi,\theta) := (-s,\phi,\theta) \in (-1,0] \times S^1 \times S^1 \subset 
M\paper \cap \uU\spine.
$$
This has the advantage that $(t,\theta) \in (-1,0] \times S^1$ now defines a
set of positively oriented collar neighborhoods of the boundary of each page.
Note that the monodromy of the bundle $\pi\paper \colon M\paper \to S^1$ cannot
be assumed trivial near the boundary, but up to isotopy we can still assume
that it takes the form $(t,\theta) \mapsto (t,\theta)$ in the above 
collars while also permuting boundary components.
With this understood, the following lemma is proved by a standard argument
(cf.~\cite{Etnyre}).

\begin{lemma}
\label{lemma:fiberLiouville}
On $M\paper$ there exists a $1$-form $\eta$ such that $d\eta$ is positive on
each fiber of $\pi\paper \colon M\paper \to S^1$ and, in the collar neighborhoods
of $\p M\paper$ with coordinates $(t,\phi,\theta)$ as defined above, 
$\eta = e^t \, d\theta$.  \qed
\end{lemma}

We can now construct a fiberwise Giroux form.  Let
$F \colon M\paper \to (0,1]$
denote a smooth function which is identically~$1$ outside of
$\uU\spine$ and takes the form $e^s f(s)$ in the collar 
coordinates $(s,\phi,\theta) \in \uU\spine$, where 
$f \colon (-1,1) \to (0,1]$ is a smooth function satisfying the conditions
\begin{itemize}
\item $f(s) = 1$ for $s \le 0$,
\item $f'(s) < 0$ for $s > 0$,
\item $f(s) = e^{-s}$ for $s$ near~$1$.
\end{itemize}
Now if $\eta$ is given by Lemma~\ref{lemma:fiberLiouville}, the expression
$$
\alpha = 
\begin{cases}
d\theta & \text{ on $M\spine$}, \\
F\eta & \text{ on $M\paper$}
\end{cases}
$$
defines a fiberwise Giroux form on~$M$.

We will use a version of the Thurston trick to turn fiberwise
Giroux forms into Giroux forms.  Given a constant $\delta \in (0,1]$,
choose a smooth function $g_\delta \colon [0,\infty) \to [0,2]$ with
\begin{itemize}
\item $g_\delta(s) = e^s$ for $s$ near~$0$,
\item $g_\delta'(s) \ge 0$ for all~$s$,
\item $g_\delta(s) = 2$ for all $s \ge \delta$,
\end{itemize}
and define from this a smooth function
$G_\delta \colon M\paper \to [0,2]$ by
$$
G_\delta = 
\begin{cases}
2 & \text{ on $M\paper \setminus \uU\spine$}, \\
g_\delta(s) & \text{ for $(s,\phi,\theta) \in \uU\spine$}.
\end{cases}
$$
Then identifying the Liouville form $\sigma$ on $\Sigma$ with its
pullback $\pi\spine^*\sigma$ on $M\spine$, we define for any
$\delta \in (0,1]$ another smooth $1$-form on~$M$ by
$$
\beta_\delta = \begin{cases}
\sigma & \text{ on $M\spine$},\\
G_\delta\, d\phi & \text{ on $M\paper$}.
\end{cases}
$$

\begin{lemma}
\label{lemma:fiberwiseGiroux}
For any fiberwise Giroux form $\alpha$, there exist constants
$\delta_0 \in (0,1]$ and $K_0 \ge 0$ such that for all constants 
$\delta \in (0,\delta_0]$ and $K \ge K_0$,
$$
\alpha_{K,\delta} := \alpha + K \beta_\delta
$$
is a Giroux form.  Whenever $\alpha$ itself is a Giroux form,
one can take $K_0 = 0$.
\end{lemma}
\begin{proof}
Observe that $\alpha_{K,\delta}$ is automatically a \emph{fiberwise} Giroux
form for all $K \ge 0$, $\delta \in (0,1]$, so we only need to show that
$\alpha_{K,\delta}$ is contact for the right choices of these constants.
Since $\beta_\delta \wedge d\beta_\delta \equiv 0$, we have
$$
\alpha_{K,\delta} \wedge d\alpha_{K,\delta} = K \left( \alpha \wedge d\beta_\delta +
\beta_\delta \wedge d\alpha\right) + \alpha \wedge d\alpha,
$$
thus it suffices to show that whenever $\delta > 0$ is sufficiently small,
\begin{equation}
\label{eqn:isContact}
\alpha \wedge d\beta_\delta + \beta_\delta \wedge d\alpha > 0.
\end{equation}
The conditions on fiberwise Giroux forms imply that $\alpha(\p_\theta) > 0$
at $\p M\paper$, so this is also true on collars of the form
$\{ s \le \delta_0 \} \subset \uU\spine$ for sufficiently
small $\delta_0 > 0$.  Assuming $0 < \delta \le \delta_0$, we shall now
show that \eqref{eqn:isContact} holds everywhere on~$M$.

On $M\spine$, $\beta_\delta \wedge d\alpha = \sigma \wedge d\alpha = 0$
since $\sigma(\p_\theta) = d\alpha(\p_\theta,\cdot) = 0$, but
$\alpha \wedge d\beta_\delta > 0$ since $\alpha(\p_\theta) > 0$ and
$d\beta_\delta = d\sigma$ is positive on~$\Sigma$.

On $M\paper$ outside of the collars $\{ s \le \delta\}$, we have
$\beta_\delta = 2\, d\phi$ and thus $d\beta_\delta = 0$, while
$\beta_\delta \wedge d\alpha = 2\, d\phi \wedge d\alpha > 0$ due to the
assumption that $d\alpha$ is positive on the fibers of~$\pi\paper$.

On the collars $\{ s \le \delta\}$, we have $\beta_\delta = G_\delta\, d\phi$,
with $G_\delta > 0$ on the interior of $M\paper$, hence
$\beta_\delta \wedge d\alpha = G_\delta\, d\phi \wedge d\alpha > 0$ again
except at $\p M\paper$.  It thus remains only to show that 
$\alpha \wedge d\beta_\delta \ge 0$, with strict positivity at $\p M\paper$.
This follows from the fact that $\alpha(\p_\theta) > 0$ on this region,
since $\alpha \wedge d\beta_\delta = g_\delta'(s)\, \alpha \wedge
ds \wedge d\phi$, where $g_\delta'(s)$ was assumed to be nonnegative and
strictly positive at $s=0$.
\end{proof}

The above implies Theorem~\ref{thm:GirouxForms}: indeed, since the space of
fiberwise Giroux forms is nonempty and convex, Lemma~\ref{lemma:fiberwiseGiroux}
shows that Giroux forms exist, and for any integer $n \ge 0$, a continuous
$S^{n}$-parameterized family of Giroux forms can be contracted through Giroux
forms.  It follows by Whitehead's theorem that the space of Giroux forms
is contractible.

\subsection{Lefschetz fibrations and Stein structures}
\label{sec:appendixStein}

In this section, we take $\Pi \colon E \to \Sigma$ to be a Lefschetz fibration
as described in the discussion preceding Theorem~\ref{thm:LefschetzStein}.
This naturally gives rise to a
spinal open book on $\p E$, with spine $M\spine := \p_h E$ and paper 
$M\paper := \p_v E$.
The fibration $\pi\paper \colon \p_v E \to S^1$ is defined as the restriction 
$\Pi|_{\p_v E} \colon \p_v E \to \p\Sigma$ after choosing an orientation preserving
identification of each connected component of~$\p\Sigma$ with~$S^1$.
Likewise, $\Pi|_{\p_h E} \colon \p_h E \to \Sigma$ defines a smooth fibration
whose fibers are disjoint unions of finitely many circles, hence it can be 
factored as
$$
\p_h E \stackrel{\pi\spine}{\longrightarrow} \widetilde{\Sigma} 
\stackrel{p}{\longrightarrow} \Sigma,
$$
where $\pi\spine \colon \p_h E \to \widetilde{\Sigma}$ is a fiber bundle with 
connected fibers over another compact oriented surface 
$\widetilde{\Sigma}$ with boundary, and $p \colon \widetilde{\Sigma} \to \Sigma$
is a smooth finite covering map.  After smoothing the corner at 
$\p_h E \cup \p_v E$, this construction gives rise to a unique isotopy class of 
spinal open books admitting smooth overlaps.

To construct Stein structures on~$E$, we will consider a special class
of almost complex structures
that always admit plurisubharmonic functions, thus giving rise to a 
distinguished deformation class of Weinstein structures.  This in turn
yields a canonical deformation class
of Stein structures due to a theorem of Eliashberg \cite{CiEl}.
Recall that a function $f \colon W \to \RR$ on an almost complex manifold
$(W,J)$ is called \defin{$J$-convex} if the $1$-form
$\lambda := - df \circ J$ is the primitive of a symplectic form that tames~$J$.
We will make repeated use of the standard fact that \emph{every} complex
structure $J$ on a compact and connected surface with nonempty boundary
admits a $J$-convex function which has the boundary as a regular level
set.  Indeed, such a function can be found by starting with a Morse
function that is $J$-convex near its critical points and post-composing with
a positive function with large second derivative (see 
e.g.~\cite[Lemma~4.1]{LatschevWendl}); in this way, one can also choose
the function's value and normal derivative at the boundary to be arbitrarily
large.

Denote by $\jJ(\Pi)$ the space of smooth almost complex structures $J$ on~$E$
that are compatible with its orientation and satisfy the following conditions:
\begin{enumerate}
\item There exists a smooth complex structure $j$ on $\Sigma$, compatible with the
given orientation, such that $\Pi \colon (E,J) \to (\Sigma,j)$ is pseudoholomorphic.
\item $J$ is integrable on some neighborhood of $E\crit$.
\item The maximal $J$-complex subbundle in $T(\p_h E)$ is preserved by some
smooth $S^1$-action on $\p_h E$ which restricts to a free and transitive
action on each boundary component of each fiber~$E_z$.
\end{enumerate}
Observe that any $J \in \jJ(\Pi)$ makes the fibers into $J$-complex curves,
with the induced orientation matching their natural orientation.
An element of $\jJ(\Pi)$ can be constructed by picking complex Morse coordinates
near $E\crit$, then choosing a suitable horizontal subbundle
outside this neighborhood which is $S^1$-invariant at $\p_h E$,
and extending the resulting complex structures
on the vertical and horizontal subbundles globally.  Since both are oriented
bundles of real rank~$2$, the space $\jJ(\Pi)$ is contractible.

Given $J \in \jJ(\Pi)$, we will say that a 
$J$-convex function $f \colon E \to \RR$ is \defin{admissible} if
the Liouville form $\lambda := - df \circ J$ restricts to a contact form
on both of the smooth boundary faces $\p_h E$ and $\p_v E$, such that
for all $z \in \Sigma$, $\p E_z \subset \p_h E$ is a union of closed Reeb 
orbits.
Observe that since $J$ is tamed by the symplectic form $d\lambda$, this
construction automatically makes the fibers symplectic, including the
pages in $\p_v E$ of the induced spinal open book at the boundary, and in
this sense one can reasonably say that $\lambda$ \emph{restricts to a 
Giroux form} on~$\p E$.  The contact condition implies that the 
induced Liouville vector field at $\p E$ is outwardly transverse to both
smooth faces, hence one can smooth the corner so that the Liouville vector
field is also transverse to the smoothened boundary, and in so doing one
can arrange for $\lambda$ to be a Giroux form for the resulting spinal
open book with smooth overlap.  
Moreover, the Liouville vector field is 
gradient-like with respect to~$f$, and one can then homotop $f$ near the
smoothened boundary through Lyapunov functions to make the smoothened
boundary a regular level set, producing a Weinstein structure uniquely
up to Weinstein homotopy.
In this way, any choice of admissible $J$-convex function $f$
determines a homotopy class of Weinstein structures which fill the contact 
structure supported by the spinal open book at the boundary.

The above discussion reduces the proof of 
Theorem~\ref{thm:LefschetzStein} to the following:

\begin{proposition}
\label{prop:PSH}
If $\Pi \colon E \to \Sigma$ is allowable, then
for every $J \in \jJ(\Pi)$, the space of admissible $J$-convex functions
is nonempty and contractible.
\end{proposition}
\begin{proof}
We proceed in three steps.

\textsl{Step~1: Existence of a fiberwise $J$-convex function.}
Given $J \in \jJ(\Pi)$,
let us call a smooth function $f \colon E \to \RR$ \defin{admissibly fiberwise
$J$-convex} if the $1$-form $\lambda := -df \circ J$ has the following
properties:
\begin{enumerate}
\item At $E\crit$, $d\lambda$ is symplectic and tames~$J$,
\item On $E \setminus E\crit$, $d\lambda$ is symplectic on every fiber,
\item For all $z \in \Sigma$, the tangent spaces to $\p E_z \subset \p_h E$
are positive for $\lambda$ but in the kernel of $d\lambda|_{T (\p_h E)}$.
\end{enumerate}
The space of admissibly fiberwise $J$-convex functions is convex and thus
contractible.  Such a function is admissibly $J$-convex if and only if
$d\lambda$ is a symplectic form taming~$J$ and $\lambda$ defines
contact forms on $\p_h E$ and~$\p_v E$.

Our first task is to construct an admissibly fiberwise $J$-convex function
$f \colon E \to \RR$.  By our assumptions on~$J$, there is a uniquely determined 
complex structure $j$ on~$\Sigma$ such that $\Pi \colon (E,J) \to (\Sigma,j)$
is pseudoholomorphic.  There is also a vertical vector field $\p_\theta$
on $\p_h E$ whose flow generates an $S^1$-action that preserves the
maximal $J$-complex subbundle
$$
\xi_h := \{ v \in T(\p_h E)\ |\ Jv \in T(\p_h E) \} \subset T(\p_h E).
$$
Note that $\Pi$ is $J$--$j$ holomorphic so $J|_{\xi_h} = \Pi^*j$, hence it is automatic
that the flow of $\p_\theta$ also preserves $J|_{\xi_h}$.  We assume
$\p_\theta$ is positive with respect to the boundary orientation of each
fiber, so $-J\p_\theta$ points transversely outwards.

To construct the desired function $f \colon E \to \RR$, we begin by
choosing for each $z \in \Sigma \setminus \Sigma\crit$ a $J$-convex
function $f_z \colon E_z \to \RR$ which at $\p E_z$ satisfies
$f_z \equiv c_z$ and $d f_z(-J \p_\theta) = \nu_z$ for some constants
$c_z , \nu_z > 0$.  We can then find a neighborhood
$\uU_z \subset \Sigma \setminus \Sigma\crit$ containing $z$ such that
$f_z$ admits an extension to a smooth function $f_z \colon E|_{\uU_z} \to \RR$
having these same properties on every fiber in $E|_{\uU_z}$.
Observe that the constants $c_z$ and $\nu_z$ can always be
made larger without changing the choice of neighborhood~$\uU_z$.
The $1$-form $\lambda_z := - df_z \circ J$ on $E|_{\uU_z}$ now satisfies
$d\lambda_z|_{T E_z} > 0$ for each $z \in \uU_z$, and its restriction
to the horizontal boundary $\alpha^h_z := \lambda_z|_{T(\p_h E)}$
satisfies $\alpha^h_z(\p_\theta) = \nu_z$, $\alpha^h_z|_{\xi_h} = 0$.

We next construct similar functions near the singular fibers.
For $z \in \Sigma\crit$, let $E\crit_z$ denote the finite set of critical
points in the fiber~$E_z$.  For each $p \in E\crit_z$, fix a neighborhood
$\uU_p \subset E$ containing $p$ on which $J$ is integrable, and choose
holomorphic coordinates $(z_1,z_2)$ identifying $\uU_p$ with a neighborhood
of~$0$ in~$\CC^2$ such that $\Pi(z_1,z_2) = z_1^2 + z_2^2$ for a suitable
choice of holomorphic coordinate near $\Pi(p) \in \Sigma$.
We use these coordinates to define a function $f_z \colon \uU_p \to \RR$ by
$$
f_z(z_1,z_2) = \frac{1}{2} \left( |z_1|^2 + |z_2|^2 \right).
$$
Then $-d f_z \circ J$ is the primitive of a positive symplectic form 
in~$\uU_p$ which tames~$J$ and restricts symplectically to the vertical 
subspaces. Since $\Pi$ is allowable, the connected components of $E_z \setminus
E\crit_z$ are all compact oriented surfaces with nonempty boundary and
finitely many punctures.  It follows that $f_z$ can be extended
so that it is $J$-convex on $E_z$ and satisfies
$f_z \equiv c_z$, $d f_z(-J \p_\theta) \equiv \nu_z$ at $\p E_z$ for some
large constants $c_z , \nu_z > 0$.  Since the $J$-convexity condition is open,
we can then extend $f_z$ over $E|_{\uU_z}$ for some neighborhood
$z \in \uU_z \subset \Sigma$ so that it has these same properties on
each fiber, and the constants $c_z, \nu_z$ can be made larger if desired 
without changing~$\uU_z$.

Since $\Sigma$ is compact, there is a finite subset $I \subset \Sigma$ such that
the neighborhoods $\uU_z$ constructed above for $z \in I$ cover~$\Sigma$.
By making the functions $f_z$ more convex near $\p_h E$,
we can then increase the constants $c_z > 0$ for all $z \in I$ so that they 
match a single constant $c > 0$, and likewise increase $\nu_z$ for
$z \in I$ to match some large number $\nu > 0$.
Choose a partition of unity $\{ \rho_z \colon \uU_z \to [0,1] \}_{z \in I}$
subordinate to the covering $\{ \uU_z \}_{z \in I}$, and define 
$f \colon E \to \RR$ by
$$
f = \sum_{z \in I} (\rho_z \circ \Pi) f_z.
$$
If $\lambda = -df \circ J$, we now have $d\lambda$ positive on all fibers,
while $d\lambda$ is symplectic and tames~$J$ near $E\crit$, and the restriction
$\alpha^h := \lambda|_{T(\p_h E)}$ to the horizontal boundary satisfies
$$
\alpha^h(\p_\theta) \equiv \nu > 0, \qquad \alpha^h|_{\xi_h} \equiv 0.
$$
It follows that $\alpha^h$ is invariant under the flow of $\p_\theta$, thus
$$
0 \equiv \Lie_{\p_\theta} \alpha^h \equiv d\alpha^h(\p_\theta,\cdot).
$$

\textsl{Step~2: The Thurston trick.}
Suppose $f \colon E \to \RR$ is any admissibly fiberwise $J$-convex
function and denote $\lambda = -df \circ J$.
Choose a $j$-convex function
$\varphi \colon \Sigma \to \RR$ which has $\p\Sigma$ as a regular level set.
Let $\sigma := -d\varphi \circ j$
denote the resulting Liouville form on~$\Sigma$.  For any constant $K \ge 0$,
consider the function
$$
F_K := f + K(\varphi \circ \Pi) \colon E \to \RR.
$$
We claim that this is admissibly $J$-convex whenever $K$ is sufficiently
large, and that this is also true for all $K \ge 0$ if $f$ itself is
admissibly $J$-convex.  Indeed, since $\Pi \colon (E,J) \to (\Sigma,j)$ is
pseudoholomorphic, we find
$$
\Lambda_K := - d F_K \circ J = \lambda + K \Pi^*\sigma.
$$
Choose a neighborhood $\uU\crit \subset E$ of $E\crit$ on which
$J$ is integrable and $d\lambda$ is a symplectic form taming~$J$.
Then for any nonzero vector $X \in T\uU\crit$,
\begin{equation}
\label{eqn:dLambda}
d\Lambda_K(X,JX) = d\lambda(X,JX) + K \, d\sigma(\Pi_*X, j\Pi_*X)
\end{equation}
is positive; here we've used the fact that $\Pi$ is $J$-$j$-holomorphic and
$d\sigma$ tames~$j$, implying that the second term is nonnegative.

To see that $d\Lambda_K$ also tames $J$ on $E \setminus \uU\crit$,
observe that the second term in \eqref{eqn:dLambda} is always nonnegative,
and is positive for $K>0$ if and only if the vector~$X$ is not vertical.
Likewise, the first term in \eqref{eqn:dLambda} is positive for nonzero vertical
vectors~$V$ and therefore also for all nonzero vectors in some open neighborhood
of the vertical subbundle.  It follows that the sum can always be made
positive if $K$ is sufficiently large.  Moreover, if $f$ is $J$-convex then
the first term is positive for any $X \ne 0$, 
and the sum is then positive for all $K \ge 0$.

It remains to check that the restrictions
$$
\alpha_K^v := \Lambda_K|_{T(\p_v E)}, \qquad
\alpha_K^h := \Lambda_K|_{T(\p_h E)}
$$
are both contact for suitable choices of $K \ge 0$.  
Let $\alpha^v := \lambda|_{T(\p_v E)}$.  Then since
$d\sigma$ vanishes on $T(\p\Sigma)$, $\Pi^*d\sigma$ vanishes on $\p_v E$,
implying
$$
\alpha_K^v \wedge d\alpha_K^v = (\alpha^v + K \Pi^*\sigma) \wedge
(d\alpha^v + K \Pi^*d\sigma) = \alpha^v \wedge d\alpha^v + K( \Pi^*\sigma \wedge d\alpha^v).
$$
Here, the second term is positive since $d\alpha^v > 0$ on fibers, thus
$\alpha_K^v$ is contact for all $K$ sufficiently large, and for all
$K \ge 0$ if $\alpha^v$ is contact; the latter is the case if $f$ is
admissibly $J$-convex.  Likewise
on $\p_h E$, we write $\alpha^h := \lambda|_{T(\p_h E)}$ and 
observe that $\Pi^*\sigma \wedge \Pi^*d\sigma$ vanishes for dimensional reasons,
while $\Pi^*\sigma \wedge d\alpha^h = 0$ since the vertical direction
is in $\ker d\alpha^h$.  Thus
$$
\alpha_K^h \wedge d\alpha_K^h = (\alpha^h + K \Pi^*\sigma) \wedge
(d\alpha^h + K \Pi^*d\sigma) = \alpha^h \wedge d\alpha^h + K( \alpha^h \wedge 
\Pi^*d\sigma).
$$
Once again the second term is positive, as $\alpha^h > 0$ in the vertical 
direction, and the contact condition for $\alpha_K^h$ follows.

\textsl{Step~3: Contractibility.}
The existence of an admissible $J$-convex function follows immediately by
combining steps~1 and~2.  Moreover, since the space of admissibly fiberwise 
$J$-convex functions is convex, step~2 implies that any continuous 
$S^n$-parameterized family of admissible $J$-convex functions is contractible, 
so the result follows via Whitehead's theorem.
\end{proof}

%\begin{remark}
%The space of Weinstein structures constructed on $E$ by the
%method above is not only connected but also contractible.  Cieliebak
%and Eliashberg have conjectured that the natural map from the space of
%Stein structures to the space of Weinstein structures is a homotopy
%equivalence---if true, then the above would imply that allowable
%Lefschetz fibrations give rise canonically to a contractible space
%of distinguished Stein structures.  The conjecture is proved at the
%level of $\pi_0$ in \cite{CiEl}, and we used this fact
%above to obtain a canonical deformation class of Stein structures.
%\end{remark}

% ==================================================================================================================

\vspace{0.3in}
%\newpage % Just to align things nicely
%\enlargethispage{0.8in}

\end{document}